\documentclass[11pt,openany,leqno]{article}
\usepackage[backref=page,colorlinks=true]{hyperref}
\usepackage{amsmath,amsthm,amsfonts,amssymb,amscd,url}
\usepackage[capitalize]{cleveref}
\usepackage{soul}

\usepackage{graphics,xcolor}
\numberwithin{equation}{section}
\input{xy} \xyoption{all}
\usepackage[latin1]{inputenc}

\usepackage{enumerate}
\usepackage{hyperref}
\usepackage{epsfig} 
\input{xy} \xyoption{all} 
\usepackage{mathrsfs}

\def \cb{\color{black} }

\def \cH{{\mathcal H}}
\def \cP{{\mathcal P}}
\def \cT{{\mathcal T}}
\def \cQ{{\mathcal Q}}
\def \cS{{\mathcal S}}
\def \cR{{\mathcal R}}
\def \TH{{\cal T}^{\scriptscriptstyle {\cal H}}}

\def \TS{{\cal T}^{\scriptscriptstyle {\cal S}}}
\def \YS{{\cal Y}^{\scriptscriptstyle {\cal S}}}
\def \XS{{\cal X}^{\scriptscriptstyle {\cal S}}}
\def \YQ{{\cal Y}^{\scriptscriptstyle {\cal Q}}}
\def \XQ{{\cal X}^{\scriptscriptstyle {\cal Q}}}
\def \TQ{{\cal T}^{\scriptscriptstyle {\cal Q}}}
\def\qand{\quad \mathrm{and}\quad}
\def\Diff{\mathrm{Diff}}
\def\Diffloc{\mathrm{Diff}_{loc}}

\def\R{\mathbb R}
\def\N{\mathbb N}
\def\I{\mathbb I}
\def\Z{\mathbb Z}
\def\cU{\mathcal U}
\def\cF{\mathcal F}

\def\det{\mathrm{det}}

\def\arr{\overleftarrow}

\newtheorem{proposition}{Proposition}[section]
\newtheorem{theorem}[proposition]{Theorem}
\newtheorem{coro}[proposition]{Corollary}

\newtheorem*{thmBintro} {Theorem B}

\newtheorem{definition}[proposition] {Definition}
\newtheorem{lemma}[proposition] {Lemma}

\newtheorem{theo}{Theorem}

\newtheorem{theoprime}{Theorem}

\newtheorem{claim}[proposition]{Claim}
\newtheorem{fact}[proposition]{Fact}

\newtheorem{corollary}[theoprime]{Corollary}

\newtheorem{question}[proposition]{Question}

\theoremstyle{remark}
\newtheorem{example}[proposition]{Example}
\newtheorem{remark}[proposition]{Remark}

\makeatletter
\def\@fnsymbol#1{\ensuremath{\ifcase#1\or *\or {**} \or ***\or
   \mathsection\or \mathparagraph\or \|\or **\or \dagger\dagger
   \or \ddagger\ddagger \else\@ctrerr\fi} }
\makeatother

\makeatletter
\renewcommand{\l@section}{\@dottedtocline{2}{3.8em}{3.2em}}
\renewcommand{\l@subsection}{\@dottedtocline{3}{3.8em}{3.2em}}
\newcommand{\subsectionruninhead}{\@startsection{subsection}{2}{0mm}{-\baselineskip}{-0mm}{\bf\large}}
\newcommand{\subsubsectionruninhead}{\@startsection{subsubsection}{3}{0mm}{-\baselineskip}{-0mm}{\bf\normalsize}}
\makeatother

\begin{document}

\title{Germ-typicality of the coexistence of  infinitely many sinks}

\author{Pierre Berger\footnote{Partially  supported  by  the  ERC  project  818737  \emph{Emergence  of  wild  differentiable  dynamical systems.}}, \; Sylvain Crovisier\footnote{Partially  supported  by  the  ERC  project 692925 \emph{NUHGD}.}, \; Enrique Pujals
}

\date{\today}

\maketitle

\abstract{
In the spirit of Kolmogorov typicality, we introduce the notion of \emph{germ-typicality}: in a space of dynamics, it encompass all these phenomena that occur for a dense and open subset of parameters of any generic parametrized family of systems.

For any $2\le r<\infty$, we prove that the Newhouse phenomenon (the coexistence of infinitely many sinks) is locally $C^r$-germ-typical, nearby a \emph{dissipative  bicycle}: a dissipative homoclinic tangency linked to a special heterodimensional cycle.

During the proof we show a result of independent interest: the stabilization of some heterodimensional cycles for any regularity class $r\in \{1, \dots, \infty\}\cup \{\omega\}$ by introducing a new  renormalization scheme.
We also continue the study of the paradynamics done in~\cite{BE15,berger2017emergence,BCP16} and
prove that parablenders appear by unfolding some heterodimensional cycles.
\cb} 

\begin{figure}[h]
\begin{center}
\includegraphics{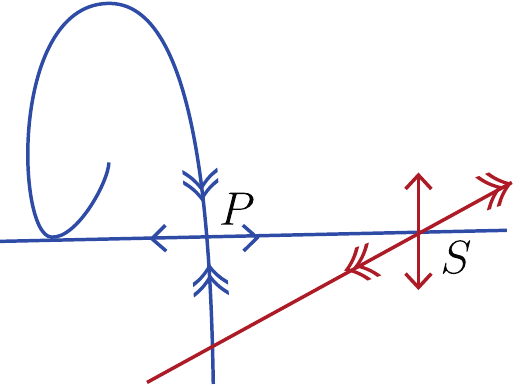}
\caption{Bicycle}\label{bicline}
\end{center}
\end{figure}

\mbox{}\begin{minipage}{0.75\textwidth}
\tableofcontents
\end{minipage}

\bigskip\bigskip

One of the most complex and rich phenomenon in differentiable dynamical systems was discovered by Newhouse \cite{Ne74,Ne79}. He showed the existence of locally Baire-generic sets of dynamics displaying infinitely many sinks which accumulate onto a Smale's horseshoe (a stably embedded Bernoulli shift).
This property is the celebrated {\em Newhouse phenomenon}. It appears in many classes of dynamics \cite{Bu97,BD99,DNP,Bu97,Duarte08,Bi20}.
Following Yoccoz,  this phenomenon provides more a lower bound on the wildness and complexity of the dynamics, rather than a complete  understanding on the dynamics. Indeed from the topological or statistical viewpoints, these dynamics are presently extremely far to be understood; it is not clear that the current dynamical paradigm would even allow to state a description of such dynamics. 
  
Since then the problem of the typicality of the Newhouse phenomenon has been fundamental. But the notion of Baire-genericity among dynamical systems is a priori independent to other notions of typicality involving probability.
That is why many important works and programs \cite{ TLY,PT93,PS95,Pa95,Pa05,Pa08,GK} wondered if  the complement of the Newhouse phenomenon could be typical in some probabilistic senses
inspired by Kolmogorov.

In his plenary talk ending the ICM 1954, Kolmogorov introduced the notion of typicality for  analytic or finitely differentiable  dynamics  of a manifold $M$.  He actually gave two definitions: one was designed to decide that a phenomenon is negligible, the other one to decide that a phenomena is typical.
He called  \emph{negligible} a phenomenon which only holds on a subset dynamics sent into a Lebesgue null subset of $\R^n$ by  a finite number of [non trivial] real valued functionals  $(\mathcal F_i)_{0\le i\le n}$ on the space of dynamics.
To decide if a phenomenon $\mathcal B$ is  \emph{typical}, he proposed to start with a dynamics $f_0$ presenting the behavior, and then to consider a deformation $f_a$ of the form 
\[f_a(z) = f_0(z)+a\cdot \phi(x,a)\; ,\]
where $\phi$ is a function of both $x$ and $a$, of the same regularity as $f$ (e.g. analytic, smooth or finitely differentiable).
Then he called the behavior $\mathcal B$ \emph{typical}, or \emph{stably realizable} if, for every $a$ small enough,
the system $f_a$ displays this behavior. This was presented as a criterion for detecting the importance of a phenomenon:
\begin{center}
\it Any type of behavior of a dynamical system for which there exists at least one example\\ of stable realization should be recognized as being important and not negligible.
\end{center}
\begin{flushright}
Kolmogorov, ICM 1954.
\end{flushright}
Recently, \cite{BE15,berger2017emergence} showed the existence of  locally Baire-generic sets of $C^r$-families of dynamics $(f_a)_{a\in \R^k}$, $r, k<\infty$,  such that for
every  $\|a\| \leq  1$, the map $f_a$ displays the Newhouse phenomenon. In particular, this showed that the complement of the Newhouse phenomenon is not typical for some interpretations of Kolmogorov's typicality. In \cite{berger2017emergence},
it has been also conjectured that some dynamics with complex statistical behavior should be typical in many senses.  

In this work we show that the Newhouse phenomenon is typical according to the following notion
inspired by Kolmogorov idea and subsequent developements \cite{Ilyashenko-nonlocal-bifurcation,KH-handbook,Mather,KZ20}:
\begin{definition}[Germ-typicality]  A behavior $\mathcal B$ is  \emph{$C^r$-germ-typical} in $\mathcal U\subset C^r(M,M)$, if
there exist a Baire-generic\footnote{i.e. a set which contains a countable intersection of open and dense sets.} set $\cR$ in the space of $C^r$-families\footnote{In \cref{sec:Preliminary},  we will precise the topological space of $C^r$-families involved.} $\hat f=(f_a)_{a\in\R}$ of maps in $\cal U$
and a lower semi-continuous function $\delta\colon \cR \to (0,+\infty)$
such that for every $\hat f\in \mathcal R$
and for all $|a|<\delta(\hat f)$, the map $f_a$ presents the behavior $\mathcal B$.
\end{definition}

To precise our setting, we consider an open set $U$ with boundary inside a surface $M$
and the space $\Diff^r_{loc}(U,M)$ of local $C^r$-diffeomorphisms from $\overline U$ to $M$.
Working inside this space, we reveal that any homoclinic configuration called \emph{bicycle} (see \cref{bicline})  is  included in the closure  of  these open sets $\cal U$:
  
\begin{definition} A local diffeomorphism  displays a \emph{bicycle} if one of its saddle point has a  homoclinic tangency and a heterocycle. A saddle point $P$ displays a  \emph{heterocycle} if $W^u(P)$ contains a projectively hyperbolic source $S$ and if  the strong unstable manifold $W^{uu} (S)$ intersects $W^s(P)$. 
The bicycle is \emph{dissipative} if the dynamics contracts the area along the orbit of $P$.
\end{definition}
Since a bicycle is a simple configuration, in many cases it may be easy to obtain,
as we will see in \cref{example simple} for the planar dynamics  $(x,y)\mapsto (x^2-2,y)$.

The main theorem of this work is the following:
\begin{theo}\label{main} 
For every $2\leq r<\infty$ and 
for every local $C^r$-diffeomorphism of surface $f\in\Diff^r_{loc}(U,M)$,  which displays a dissipative bicycle, there exists  a (non empty) open set $\mathcal U^r\subset \Diff^r_{loc} (U, M) $ whose closure contains $f$
and where the Newhouse phenomenon is   $C^r$-germ-typical. 
\end{theo}

Following Kolmogorov viewpoint, this theorem strengthens the evidence of the importance of the Newhouse phenomenon.
As we mentioned previously, \cite{BE15} discovered the stable coexistence of infinitely many sinks for a generic subset of an open set
of parametrized families, i.e. showed that the Newhouse phenomenon can not be neglected when one crosses a region of the space of systems
along some ``well-chosen directions". In comparison \cref{main} goes one step further and establishes the typicality of the Newhouse phenomeon in a sense which only depends on a neighborhood inside the space of systems (and not on the neighborhood of a specific family). 

\cb

For the sake of clarity, we restrict the scope of the present work to the case of surface local diffeomorphisms and to a notion of germ-typicality involving only one-parameter families. However we are confident that our result could be generalized to a broader setting:  for instance \cite{berger2017emergence} applies also to  diffeomorphisms in higher dimension and to any finitely dimensional parameter families. 

\medskip
\subsection*{Locus of robust phenomena: stabilization of heterodimensional cycles}
Another main point of the present work is to bring to light
a very simple configuration nearby which germ-typicallity of the Newhouse phenomenon holds true. The idea to associate a homoclinic configuration to a phenomenon goes back to Newhouse.
In  \cite{Ne70}, he first showed that it is possible to get a (non-empty) open set of diffeomorphisms of surface exhibiting homoclinic tangencies (these diffeomorphisms exhibit \emph{$C^2$-robust homoclinic tangencies}), and then in \cite{Ne74}  that this open set can encompasses a Baire-generic subset formed by dynamics displaying the Newhouse phenomenon (infinitely many attracting cycles). To obtain such open sets of diffeomorphisms with robust homoclinic tangencies, Newhouse considered horseshoes with large fractal dimension (large thickness in his own nomenclature). Later, in \cite{Ne79}, Newhouse proved that from [the configuration defined by] a homoclinic tangency, a perturbation of the dynamics displays a robust homoclinic  tangency (see \cref{Newhouse}). 

For local diffeomorphisms, these thick horseshoes  can be replaced by
more topological object, called \emph{blenders}. They  were introduced by Bonatti and Diaz \cite{BD96} for diffeomorphisms in
dimension larger than or equal to three and can be recasted in the context of local diffeomorphisms of surface as hyperbolic compact sets such that the union of their local unstable manifolds covers $C^r$-robustly a (non-empty open set of the surface (see \cref{def:blender}).  In the same spirit as Newhouse's work, one can wonder, nearby which
homoclinic configurations blenders appear. Bonatti, Diaz and Kiriki \cite{BDK} proved that heterodimensional cycles (which, in the case of local diffeomorphisms of surface correspond to cycles between a saddle and a source) play that role when one considers the $C^1$-topology: a $C^1$-perturbation of the heterodimensional cycle generates open sets of dynamics exhibiting blenders and $C^1$-robust heterodimensional cycles.  In the present paper, we extend this result to the context of more regular dynamics:
\begin{thmBintro} \label{robust heterointro}
For every $1\le r\le \infty$ or $r=\omega$,
consider $f\in \Diffloc^r(U,M)$ exhibiting a heterocycle associated to a saddle $P$.  Then there exists $\tilde f$, that is  $C^r$-close to $f$, with a basic set $K$ containing the hyperbolic continuation of $P$,  and which has a $C^r$-robust heterocycle.
 \end{thmBintro} 
 While communicating our result, Li and Turaev have informed us that they independently achieved to prove a more general version of Theorem \ref{robust hetero} for higher dimensional systems, by using different techniques \cite{LiTur}.
 Diaz and Perez have also recently obtained~\cite{DP20} a similar stabilization
of heterodimensional cycles for $C^r$-diffeomorphisms in dimension $3$,
assuming in addition that one of the periodic points exhibits a homoclinic tangency.

\subsection*{Renormalization nearby heterocycles}   
  In order to prove Theorem \ref{robust hetero} (in \cref{sec:robust robust hetero}), we first show in \cref{pingpong}  that nearby heterocycles there are heterocycles satisfying an additional property. These configurations are called \emph{strong heterocycles} and are defined in \cref{def hetero}. Then  \cref{PPaffine}   introduces a renormalization nearby strong heterocycles to obtain nearly affine blenders.

This renormalization consists in selecting two inverses branches $g^+$ and $g^-$ of larges iterates of the dynamics, which are defined on boxes nearby the heterocycle and then to rescale  $\mathcal R g^- =\phi^{-1}\circ g^- \circ \phi$,
$\mathcal R g^+ =\phi^{-1}\circ g^+ \circ \phi$ the two latter branches via a same coordinate change $\phi$.
The maps $\mathcal R g^-, \mathcal R g^+$ are close to affine maps and define a blender, which will be called \emph{nearly affine blender}, see \cref{def:affine blender}.

Theorem \ref{robust hetero} is restated more precisely in \cref{sec:blender}.  
 \cref{pingpong,PPaffine} are proved in respectively \cref{s.strong,s.blender}. This renormalization is one of the main technical novelty of the present work. It is further developed to obtain \cref{theosectionparadense} (in \cref{sec:parablender}), a parametric counterpart   of   Theorem  \ref{robust hetero}. \cref{theosectionparadense} is essential to prove Theorem \ref{main}. It states that nearby paraheterocycles there are nearly affine parablenders. These are objects of paradynamics.
 
\subsection*{Paradynamics}
To explain the role of these parametric blenders we have to go back to the paper \cite{BE15}: it considered parameter families of 
local diffeomorphisms on surfaces and introduced the notion of \emph{paratangencies}: a homoclinic tangency that is ``sticky'' (or  unfolded in ``slow motion"). That phenomenon implies that the attracting periodic points created by the unfolding of the tangency have ``a long life in the parameter space".
Moreover, if any perturbation of a parameter family still exhibits a dissipative homoclinic paratangency for all parameters
(in other words the family exhibits \emph{robust homoclinic paratangencies}, the analog in the space parameter families of the robust homoclinic tangencies in the space of local diffeomorphisms) then, after small perturbation,  the new family displays infinitely many attracting periodic points for all parameters  (see \cref{prop34,prop36}). 
      
To provide robust  paratagencies, \cite{BE15} introduced a parametric version of the blenders, called \emph{$C^r$-parablenders},
see \cref{def.parablender}. To grasp the idea behind this notion, first  recall that any hyperbolic compact set of a map has a unique  continuation for a nearby system. Any point in the hyperbolic set has a unique continuation as well (see \cref{sec:Preliminary} for details) and the same  holds true for its local stable and unstable manifolds. When the  parameter family is of class  $C^r$,  the continuation of a point defines a curve of class $C^r$. The key property of a $C^r$-parablender, is that
for an open set of parametrized points in the surface,
of the local unstable manifold of the parablender moves in slow motion with respect to the parametrized point.
This property can be push forward to the unfolding of homoclinic tangencies and allows to
create robust homoclinic paratangencies.
For that purpose, it is easier to assume that the collection of local unstable manifolds covers a source
homoclinically linked to the parablender.

In \cite{BCP16}, the notion of parablender has been recasted:
parameter families of maps naturally induce an action on $C^r$-jets
and the parablenders can be viewed as
blenders for this dynamics on the space of jets. This viewpoint allowed us to systematize the construction of
parablenders: in  \cite{BCP16},
using Iterated Function Systems, a special type of parablenders called \emph{nearly affine parablenders} (see \cref{parablenders defi}) is introduced.
 
In the present paper, we tried to follow Newhouse's approach and looked for a simple bifurcation that generates ``robust paratangencies". According to \cite{berger2017emergence}, it suffices to obtain a parablender covering a source and linked to a dissipative  homoclinic tangency. Similarly to \cite{BDK}, one can wonder if the parametric  unfolding of a heterodimensional cycle may generate a parablender. We answer by proving that the unfolding of a homoclinic tangency related  to a  heterocycle (a \emph{bicycle}) is the sough configuration which produces
robust paratangencies.

 To precise, first we prove that combining a homoclinic tangency with the heterocycle, one obtains alternate chain of heterocycles (a chain of heterocycles involving saddles with negative eigenvalues, see \cref{def.Nchain}). The unfolding of that special chain produces a \emph{paraheterocycle} (a heterocycle that is unfolded in ``slow motion'', see \cref{def.paraheterocycle} and theorem \ref{chain}) and which then gives birth to nearly affine parablenders (see theorem \ref{theosectionparadense}) using the aforementioned renormalization technique. 

\subsection*{Open problems}
Paradynamics has been useful to prove that several complex and interesting phenomena are robust along a locally Baire-generic set of families of dynamics, see \cite{Be17per,IS17,BR21att,BR21}. 
The tools brought by our work should enable to show the $C^r$-germ-typicality of these phenomena. 

Note that if a behavior $\mathcal B$ is $C^r$-germ-typical in $\mathcal U$ then it occurs on an open and dense set of parameters for a Baire-generic set of $C^r$-families $(f_a)_a$ of dynamics $f_a\in \cal U$. But it does not imply that the Lebesgue measure of this open and dense set of parameters is full. In particular,  it remains open whether the Newhouse phenomenon is locally typical with respect tosome  interpretations of 
Kolmogorov typicality given by \cite{KH-handbook}, \cite{BE20} or  \cite[Chapter 2, section 1]{Ilyashenko-nonlocal-bifurcation}. The latter is slightly stronger than:
\begin{definition}[Arnold prevalence (soft version)]
For $r, k\ge 1$, a behavior $\mathcal B$ is  $C^r$-$k$-Arnold prevalent in $\mathcal U\subset C^r(M,M)$, 
if there exists a Baire-generic set $\mathcal R$ of $C^r$-families $(f_a)_{a\in \R^k}$ formed by maps $f_a\in \cal U$ such that for every $(f_a)_a\in \mathcal R$, for Lebesgue almost every parameter $a\in \R^k$ , $f_a$ presents the behavior $\mathcal B$. 
\end{definition}

Another notion of typicality has been introduced by Hunt, York and Sauer \cite{HSY}, and then developed by Kaloshin-Hunt in \cite{KH-handbook}; it was used by Gorodeski-Kaloshin \cite{GK} to study the typicality of the Newhouse phenomenon,  
but leaves open the problem of the typicality of  Newhouse phenomenon following the latter notions\footnote{The original notion of typicality defined by \cite{HSY} is defined for Banach spaces; its counterpart for Banach manifolds (such as the space of dynamics on a compact manifold) is so far not unique (there is no version of this notion which is invariant by coordinate change, contrarily to germ-typicality or Arnold prevalence).}.

Finally let us emphasize that the important Arnold-prevalence or the germ-typicality of the Newhouse phenomenon are open for the $C^\infty$ or analytic topologies. Hopefully the tools developed in this present work seem to us useful to progress on these important problems.

\section{Concepts involved in the proof}\label{concepts}
In this section we state the main results which are used to obtain \cref{main}. 

In \cref{sec:Preliminary} we recall classical definitions about hyperbolicity in the particular context of local diffeomorphisms.
In \cref{sec:homocycle} and\cref{sec:hetero} we recall the concepts of homoclinic tangency and heterodymensional cycle between fixed points with different indices and the classical results of Newhouse and Bonatti-Diaz associated to these bifurcations.
In \cref{sec:blender} we recall the notions of blenders and nearly affine blenders and we state the main theorem that relates cycles and blenders (\cref{robust hetero}).
In \cref{bicycle}, we state precisely the definition of bicycle (that combines a homocycle and a heterocycle) and
we show in \cref{robust-bicycle} that from bicycles one can obtain robust bicycles (it is worth to mention that this is done in any  $C^r$-regularity including the analytic one).

In \cref{sec:parahetero} and \cref{sec:parablender} we give the parametric version of the previous results. In \cref{sec:parahetero} we introduce the notion of paraheterocycle and explain how by unfolding heterocycles associated to saddles with negative eigenvalues one can obtain a paraheterocycle (\cref{chain}). In \cref{sec:parablender} we introduce the notions of affine and nearly affine parablenders and explain how they emerge from paraheterocycles (\cref{theosectionparadense}).

\subsection{Preliminaries}\label{sec:Preliminary}
In the following $M$ is a compact surface, $U$ an open subset whose boundary is a smooth submanifold
and $\Diffloc^r(U,M)$ for $r\in \N\cup \{\infty\}$, denotes the restrictions to $U$ of $C^r$-map
$f\colon \overline U\to M$ whose differential $D_xf$  is invertible at every $x\in \overline U$.
 Endowed with the $C^r$-topology,
this is a Baire space.

For some results,
one will also assume that $M$ is a real analytic surface
and let $\tilde M$ be a complex extension.
One then considers the space $\Diffloc^\omega(U,M)$
of real analytic maps
endowed with the analytic topology defined as the inductive limit
of the spaces of holomorphic maps defined on neighborhoods of $M$
in $\tilde M$.

Now let us precise the space of $C^r$-families
parametrized by the interval $\I=(-1,1)$.
For the sake of clarity, we will focus only on the space
$\mathcal{D}^r(\I \times U,M)$ of families $(f_a)_{a\in \I}$
which are the restriction of a map
$(a,x)\mapsto f_a(x)$ of class $C^r$
on $\overline \I \times \overline U$, that we endow with the uniform $C^r$-topology. 
However all our arguments will be also valid for the smaller space $C^r(\overline \I,\Diff^r_{loc} (U, M))$ endowed with the topology of $C^r$-maps from $\overline \I$ into $\Diff^r_{loc} (U, M)$.

An \emph{inverse branch} for  $f\in \Diffloc^r(U,M)$ is the inverse of a restriction $f|V$ of $f$ to a domain $V\subset U$ such that $f^n|V$ is a diffeomorphism onto its image. 
\medskip

A compact set $K$ is \emph{(saddle) hyperbolic} for $f$ if it is $f$-invariant (i.e. $f(K)= K$) and  there  exists a continuous, $Df$-invariant subbundle $E^s$ of  $TM|K$ which is uniformly contracted and normally uniformly expanded. More precisely,
there exists $N\geq 1$ satisfying: 
\[\|D_zf^N|E^s_z\|<1/2 \quad \text{and}\quad \|p_{E^{s\bot}}\circ  D_zf^N(v)\|\ge 2\|v\|  \; ,\quad \forall z\in K,\; v\in E^{s\bot}_z\; ,\]
where $E^{s\bot}$ is the subbundle of $TM|K$ equal to the orthogonal complement of $E^s_z$ and  $p_{E^{s\bot}}$ the orthogonal projection onto it. The hyperbolic set $K$ is a \emph{basic set} if it is transitive and locally maximal. Then  $K$ is equal to the closure of its subset of periodic points. 

Any point $x\in K$ has a stable manifold $W^s(x)$ (also denoted $W^s(x;f)$) which is an injectively immersed curve.
The map $f|K$ being in general not injective, a single point $x\in K$ has in general as many unstable manifolds as preorbits $\underline x$. We denote such a submanifold  by $W^u(\underline x)$, or $W^u(\underline x; f)$.  The space of preorbits $\underline x$ is denoted by $\overleftarrow K:= \{\underline x=(x_i)_{i\le 0}\in K^{\Z^-}:\; f(x_{i-1})=x_i\}$. The space $\overleftarrow K$ is canonically endowed with the product topology. The zero-coordinate projection is denoted by $\pi_f\colon \overleftarrow K\to M$; it semi-conjugates the shift dynamics $\sigma$ on $\overleftarrow K$ with $f$. 

It is well known (see for instance \cite{BR13}) that a hyperbolic compact set is \emph{$C^1$-inverse limit stable}: for every $C^1$-perturbation $f'$ of $f$,   there exists a (unique) map $\pi_{f'}\colon \overleftarrow K\to M$ which is $C^0$-close to $\pi_{f'}$ and so that:
\[\pi_{f'}\circ \sigma =   f'\circ \pi_{f'}\;.\]
The image $K_{f'}:= \pi_{f'}(\overleftarrow K)$ is also a hyperbolic set.  Note that $K_f=K$. Also $K_{f'}$ is called the \emph{hyperbolic continuation}  of $K$. 

Two basic sets  are \emph{(homoclinically) related} if there exists an unstable manifold of the first which has a transverse intersection point with a stable manifold of the second, and vice-versa. Then by the Inclination Lemma, the local unstable manifolds of one basic set are dense in the unstable manifolds of the other. 
\medskip

An $f$-invariant compact space is \emph{projectively hyperbolic expanding} if there exists
a continuous $Df$-invariant subbundle $E^{cu}$ of $TM|K$ which is uniformly expanded and normally uniformly expanded. More precisely,
there exists $N\geq 1$ satisfying: 
\[\|D_zf^N|E^{cu}_z\|>2 \quad \text{and}\quad \|p_{E^{cu\bot}}\circ  D_zf^N(v)\|\ge 2\cdot \|v\|\cdot   \|D_zf^N|E^{cu}_z\|  \; ,\quad \forall z\in K,\; v\in E^{cu\bot}_z.\]
If it is transitive and locally maximal, it is equal to the closure of its subset of periodic points.
To any $\underline x\in \overleftarrow K$, one associates a strong unstable manifold $W^{uu}(\underline x)$
as the set of points which converge to the orbit of $\underline x$ in the past transversally to the bundle $E^{cu}$.
\medskip

 A saddle periodic point $P$ of period $p\ge 1$, is \emph{dissipative} if $|\det D_Pf^{p}|<1$.

A source periodic point $S$ is \emph{projectively hyperbolic} if the tangent space at $S$ split into two $Df$-invariant directions, $T_SM=E^{cu}\oplus E^{uu}$, the direction  $E^{cu}$ --called \emph{center unstable}-- being less expanded than the direction $E^{uu}$ --called \emph{strong unstable}. Its strong unstable
manifold $W^{uu}(S)$ is the set of points which converge to the orbit of $S$ in the past
in the direction of $E^{uu}$.

\subsection{Homocycle}\label{sec:homocycle} Given $f\in \Diffloc^r(U,M)$, a saddle periodic point $P\in U$ has a \emph{homoclinic tangency} or \emph{homocycle} for short, if its stable manifold has a non-transverse intersection point $T\in U$ with its unstable manifold.
\begin{equation}\tag{Homocycle} \exists T\in TW^s(P)\cap TW^u(P)
\; .\end{equation}

More generally, a basic set $K\subset U$  has a homoclinic tangency if there exist $P\in K$
and $\underline Q\in \overleftarrow K$ (not necessarily periodic) such that $W^s(P)$ is tangent to $W^u(Q)$. 
A basic set $K$ has a \emph{$C^r$-robust  homoclinic tangency} if for every $C^r$-perturbation of the dynamics, the hyperbolic continuation of $K$ still has a homoclinic tangency. 
If $r\geq 2$ and if the phase space is a surface, the tangency $T$ is \emph{quadratic},
if  the curvature of $W^s(P)$ and $W^u(\underline Q)$ at $T$ are not equal.

\begin{figure}[h]
\begin{center}
\includegraphics{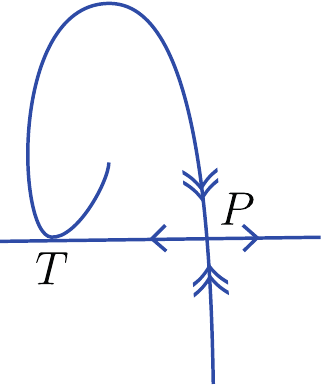}
\caption{Homocycle}\label{homocycle}
\end{center}
\end{figure}

Here is a famous theorem by Newhouse~\cite{Ne79}, which stabilizes the homoclinic tangencies.
\begin{theorem}[Newhouse]\label{Newhouse} For $2\leq r\leq \infty$ or $r=\omega$, consider $f\in \Diffloc^r(U,M)$  and a saddle periodic point $P$ exhibiting a homoclinic tangency $T$.
Then there exists $\tilde f$ $C^r$-close to $f$, with a basic set $K$
containing the hyperbolic continuation of $P$,  and which has a $C^r$-robust homoclinic tangency.

\end{theorem}

The open set $\mathcal N^r $ of dynamics displaying a $C^r$-robust homoclinic tangency is called the \emph{Newhouse domain}. We denote by $\mathcal N^r(P)\subset \mathcal N^r$ the open set of dynamics for which the hyperbolic continuation of $P$ belongs to a basic set displaying a $C^r$-robust homoclinic tangency.   By the Inclination Lemma, the stable and unstable manifolds of $P$ are dense in the stable and unstable sets of $K$. Thus a $C^r$-small perturbation of any dynamics in $\mathcal N^r(P) $ creates a homoclinic tangency for $P$. This proves:
\begin{proposition}\label{denistehomocylcle}
For every $1\le r\le \infty$ or $r=\omega$, there exists a $C^r$-dense set in $ \mathcal {N}^r(P)$, made by maps for which the hyperbolic continuation of $P$ has  a homoclinic tangency. 
\end{proposition}
Let $\mathcal N_{diss}^r(P)\subset \mathcal N(P)$ be the open set formed by dynamics for which the hyperbolic continuation of $P$ is dissipative.
As  a periodic sink of arbitrarily large period can be obtained by a small perturbation of a dissipative homoclinic tangency, the latter proposition then implies the Baire-genericity in $\mathcal N_{diss}^r(P)$  of dynamics exhibiting a Newhouse phenomenon (see  \cite{Ne74} for more details). 
 
\subsection{Heterocycles}\label{sec:hetero}

In the present section we first recast for the case of surface endomorphisms, the notion of heterodimensional cycle  introduced in \cite{D,BD96}, and present two stronger versions of it called \emph{heterocycle} and \emph{strong heterocycle}. 

\begin{definition}\label{def hetero}
A map $f\in \Diffloc(U,M)$ displays a \emph{heterodimensional cycle}  if 
it has a  saddle periodic point $P$ and 
a periodic source $S$   such that 
$W^u(S)$ intersects $W^s(P)$ and $S$ is  in $W^u(P) $: 
\begin{equation}\tag{Heterodimensional cycle} S\in W^u(P) \quad \text{and}\quad W^s(P)\cap W^{u}(S)\not= \varnothing\; .\end{equation}
The heterodimensional cycle forms a \emph{heterocycle} if the source is projectively hyperbolic and  $W^{uu}(S)$ intersects $W^s(P)$:
\begin{equation}\tag{Heterocycle} S\in W^u(P) \quad \text{and}\quad W^s(P )\cap W^{uu}(S)\not= \varnothing\; .\end{equation}
This  heterocycle is \emph{strong} if furthermore $W^{uu}(S)$ contains $P$:
\begin{equation*}\tag{Strong heterocycle} S\in W^u(P)  \qand
P\in 
W^{uu} (S)
\; .\end{equation*}
\end{definition}
We will see in \cref{pingpong}  that any map displaying a heterocycle  can be smoothly perturbed to display a strong heterocycle between a saddle point $P'$ homoclinically related to the initial one  $P$, and the initial source $S$.

\begin{figure}[h]
\begin{center}
\includegraphics{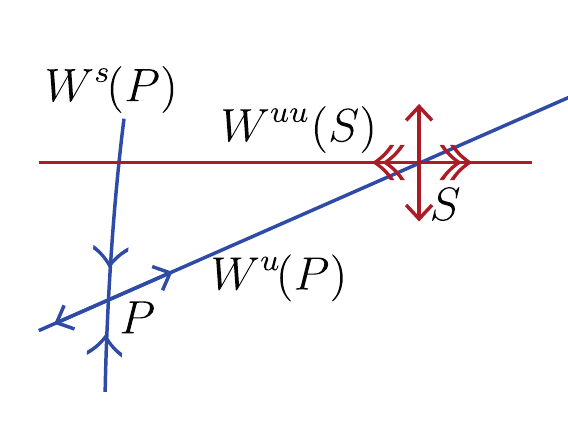}
\caption{Heterocycle for a surface map. \label{fig1}}
\end{center}
\end{figure}

A heterocycle is a one-codimensional phenomenon. To show its local density, we shall generalize it as follow. A basic set $K$ and a projectively hyperbolic periodic source $S$ of a surface map display a \emph{heterocycle} if there exists $P\in K$ (not necessary periodic) such that $ W^s(P)\cap W^{uu}(S)\not= \varnothing$ and there exists $\underline P\in \overleftarrow K$ such that $P=\pi_f(\underline P)$ and $S\in W^u(\underline P)$. 
The \emph{heterocycle is $C^r$-robust} if for every $C^r$-perturbation of the dynamics, the hyperbolic continuations of $K$ and $S$ still have a heterocycle. 
The $C^r$-open set of surface maps which display a robust heterocycle is called the \emph{Bonatti-Diaz domain} and is denoted by $\mathcal {BD}^r$.   We denote by $\mathcal {BD}^r(P,S)\subset \mathcal {BD}^r$ the open set of dynamics for which the hyperbolic continuation of $P$ belongs to a basic set displaying a $C^r$-robust heterocycle with the continuation of $S$. 

\subsection{Blenders}\label{sec:blender}

Let us again consider a robust heterocycle between a basic set $K$ and a source $S$.
As by a perturbation of the dynamics, $S$ can be moved independently of $K$ and its unstable manifold, this implies that $K$ must be a \emph{blender}:
 \begin{definition}[$C^r$-Blender]\label{def:blender}
A \emph{$C^r$-blender} for $f\in \Diffloc^r(U,M)$ is a basic set $K$ such that the union of its local unstable manifolds has $C^r$-robustly non-empty interior: there exists  a continuous family of local unstable manifolds whose union contains   a non-empty open set $V\subset U$ and the same holds true for their  continuations for any $C^r$-perturbations  $\tilde f$  of $f$.

The set $V$ is called an \emph{activation domain} of the blender $K$. 
 \end{definition}
 
 As the periodic points are dense in $K$, the unstable manifolds of periodic points are also dense in the activation domain. Hence for a small $ C^r$-perturbation supported by a small neighborhood of the blender, there exists a
 periodic point whose unstable manifold contains the source, defining a heterocycle. This proves the following counterpart of \cref{denistehomocylcle}:
\begin{proposition}\label{denisteheterocylcle}
For every $1\le r\le \infty$ or $r=\omega$, there exists a $C^r$-dense set in $ \mathcal {BD}^r(P,S)$ made by maps
for which the hyperbolic continuation of $P$ and $S$ have  a heterocycle.
\end{proposition}
 
Bonatti and Diaz have introduced the notion of blender
and obtained the first semi-local constructions of robust heterocycles \cite{BD96}.
\begin{question}
All the known $C^r$-blenders are also $C^1$-blender.
\emph{Is  $\mathcal {BD}^r$ equal to $\mathcal {BD}^1$?}
\end{question} 
 
 The following notion has been introduced in \cite{BCP16} and will play a key role in a renormalization that we will perform nearby heterocycles. 
  \begin{definition}[Nearly affine blender]\label{exam.blender}\label{def:affine blender}
For $r\in [1,\infty)$, $\Delta>1$, $x_0\in(-2,2)$,
$\delta>0$,
$f$ has a \emph{$\delta$-$C^r$-nearly affine blender} with contraction $\Delta^{-1}$ if there is a $C^r$-chart $H \colon \R^2\hookrightarrow  M$
such that:
\begin{itemize}
\item[--] there is an inverse branch $g^+$ of an iterate $f^{N^+}$ of $f$ such that $\cR g^+ := H^{-1}\circ g^+\circ H$ is well defined on $[-2,2]^2$ and is  $\delta$-$C^r$-close to $(x,y)  \mapsto (x_0, \Delta  (y-1)+1)$;
\item[--] there is an inverse branch $g^-$ of an iterate $f^{N^-}$ of $f$ such that $\cR g^- := H^{-1}\circ g^-\circ H$ is well defined on $[-2,2]^2$ and is  $\delta$-$C^r$-close to $(x,y)  \mapsto (x_0, \Delta  (y+1)-1)$. 
\end{itemize}
\end{definition}
 Observe that the   maximal invariant set  of  the  map:
 \[(\cR g^+)^{-1}\sqcup (\cR g^-)^{-1}: \cR g^+([-2,2]^2)\sqcup (\cR g^-([-2,2]^2) \to \R^2\]
  is a basic set $K$. The following is easy, see for instance \cite[Section 6]{BCP16} for details. 
\begin{proposition}
\label{ are blender}
For every  $\Delta>1$  close to $1$,  $x_0\in (-2,2)$ and $\eta\in(0,1)$, if   $\delta>0$ is sufficiently small,  then the set $K$ is a $C^1$-blender and $(-2,2)\times  \left[-1
+\eta,
1-\eta\right] $ is  an activation domain.
\end{proposition}
\medskip

In   \cref{Proof robust hetero} we will prove the following analogous of Newhouse \cref{Newhouse},
which stabilizes the heterocycles.
It will be obtained by introducing a renormalization for a perturbation of $f$ leading to a nearly affine blender.
\begin{theo}\label{robust hetero}
For every $1\le r\le \infty$ or $r=\omega$,
consider $f\in \Diffloc^r(U,M)$ exhibiting a heterocycle formed by a saddle $P$ and a projectively hyperbolic source $S$. 
Then for every $\delta>0$ and any  number $\rho\le r$,   there exists $\tilde f$, $C^r$-close to $f$,
such that $P_{\tilde f}$ is homoclinically related to a $\delta$-$C^\rho$-nearly  affine blender whose activation domain contains $S_{\tilde f}$. 
\end{theo}

\begin{question} 
To what extend the previous results generalize to heterodimensional cycles?
\end{question}
In that direction, \cite{BDK} proved for diffeomorphisms that it is possible to stabilize by $C^1$-perturbation
any \emph{classical} heterodimensional cycle between saddles whose stable dimension differs by one,
provided that at least one of the saddle involved in the cycle belongs to a nontrivial hyperbolic set.  An analogue in any regularity class is done in  \cite{LiTur}.

\subsection{Bicycles and robust bicycles}\label{bicycle}
Let us precise the definition of bicycle mentioned in the introduction:
\begin{definition}
A saddle $P$ and a projectively hyperbolic source $S$
display a \emph{bicycle} if they form a heterocycle and if $P$ has a homocycle.
The bicycle is \emph{dissipative} if the orbit of $P$ is dissipative.
\end{definition}

The notion of bicycle can be extended to basic sets.

\begin{definition} A basic set for $f\in \Diffloc^r(U,M)$ displays a \emph{$C^r$-robust bicycle} if it displays a $C^r$-robust homocycle and forms a $C^r$-robust heterocycle with a projectively hyperbolic source.
\end{definition}

It is easy to build a bicycle by perturbation of some explicit example: 
\begin{example}\label{example simple} For every $r\ge 2$, the map $f:= (x,y) \in \R^2\mapsto (x^2-2,y)$ is $C^r$-accumulated by maps $f_\varepsilon$ exhibiting  a bicycle. 
Hence by \cref{main},  there is an open set of $C^r$-perturbations $\mathcal U^r$ of $f_\varepsilon$ in which the coexistence of infinitely many sinks is $C^r$-germ-typical.  
\end{example}

\begin{proof}[Proof of \cref{example simple}]
 First, we choose the parameter $a$ close to $-2$ such that the map $g(x)=x^2+a$
 admits two homoclinically related repelling  periodic points $s$, $p$, the orbit of the critical point contains $p$
 (there exists $n\geq 1$ such that $g^n(0)=p$) and belongs to the unstable set of $p$
(there exists a sequence of backward iterates of $0$ which accumulates on $p$): usually such a parameter $a$ is called a {\em Misiurewicz parameter}).

Then we consider a function $\rho$ close to $1$ which is equal to $1+\varepsilon$ on a small neighborhood of the orbit of $s$ and
to $1-\varepsilon$ in a small neighborhood of the orbit of $p$.  
We now consider the following small perturbation of $f$:
 \[f_{a,\varepsilon} (x,y)\mapsto (x^2+a,  \rho(x) y).\]
 Observe that it has  a projectively hyperbolic source $S := (s,0)$ and  dissipative saddle point $P:= (p,0)$, such that the unstable manifold of each point intersects the other point and the image of the critical point still is preperiodic.  One now can perform a small perturbation in a neighborhood of the critical point that makes the map a local diffeomorphism and that also preserves the image of the critical point, which then defines a homoclinic tangency of $p$.
In a such way,  one  obtains a map with  a bicycle involving $P$ and $S$. 
\end{proof}

Similarly to \cref{denistehomocylcle} and \cref{denisteheterocylcle} we have:
\begin{proposition}\label{densitebicycle}
For every $1\le r\le \infty$ or $r=\omega$,
consider an open set of maps $f\in \Diffloc^r(U,M)$
displaying a $C^r$-robust bicycle involving a saddle $P$ and a projectively hyperbolic source $S$.
It contains a $C^r$-dense subset of maps for which the hyperbolic continuation of $P$ and $S$ form a bicycle.
\end{proposition}

Combining  Theorems \ref{Newhouse} and  \ref{robust hetero}, one can stabilize the bicycles:

\begin{corollary}\label{robust-bicycle} For $2\leq r\leq \infty$ or $r=\omega$, consider
$f\in \Diffloc^r(U,M)$ and a saddle $P$ exhibiting a bicycle.
Then there exists $\tilde f $, $C^r$-close to $f$, with a hyperbolic basic set $K$ containing the hyperbolic continuation of $P$ which exhibits a $C^r$-robust bicycle.
\end{corollary}

\subsection{Paraheterocycles}\label{sec:parahetero} 
Let us fix $1\leq r \leq \infty$,
and a $C^{r}$-family $(f_a)_{a\in \R}$ of local diffeomorphisms $f_a\in \Diffloc^r(U,M)$. 
\paragraph{Hyperbolic sets for  families of dynamics}
 It is well known that if $f_0$ has a hyperbolic fixed point $P$, then its hyperbolic continuation 
$(P_a)_{a\in I} $  is a $C^r$ function of the parameter $a$
on a neighborhood $I\subset \R$ of $0$.  More generally, if  $K$ is a hyperbolic set for $f_0$, with $\arr K$ the inverse limit of $K$, its hyperbolic continuation  $(K_a)_{a\in I}$  by  the range $K_a=\pi_a(\arr K)$ of a family of maps  $\pi_a:= \pi_{f_a}: \arr K \to M$    (see Section~\ref{sec:Preliminary}) with  the following  regularity:
\begin{proposition}[see Prop 3.6 \cite{BE15}]\label{continuation}
There exists a neighborhood $I$ of $0$ where $(\pi_a)_{a\in I}$ is well defined.
For any $\underline  z\in \arr K$, the map $a\in I \mapsto \pi_a(\underline z)$ is of class $C^r$ and depends continuously on $\underline z$ in the $C^r$-topology.
\end{proposition}

The local stable and unstable manifolds $W^s_{loc} (z; f_a) $ and $W^u_{loc} (\underline z; f_a) $  are canonically chosen so that they depend continuously on $a$, $z$ and $\underline z$ in the $C^r$-topology
(see Prop~3.6 in~\cite{BE15}).
They are called the \emph{hyperbolic continuations} of $W^s_{loc} (z; f_0) $ and $W^u_{loc} (\underline z; f_0) $ for $f_a$.

\begin{definition}[Paraheterocycle]\label{def.paraheterocycle}
Given $0\leq d\leq r$,
the family $(f_a)_{a\in \R}$ displays a \emph{$C^d$-paraheterocycle} at $a_0$ if there exist
a heterocycle for $f_{a_0}$ involving a saddle $P$ and
a projectively hyperbolic source $S$ whose hyperbolic continuations
satisfy for some $N\geq 0$
\begin{equation}\label{e.paracycle}
d(S_a,f^N_a(W^{u}_{loc}(P_a)))=o(|a-a_0|^{d'}), \quad \text{ for any integer $0\leq d'\leq d$.}
\end{equation}
We say it is a \emph{strong $C^d$-paraheterocycle} if furthermore $P$, $S$ form a strong heterocycle.
\end{definition}
Note that if $f_{a_0}$ has a heterocycle then $(f_a)_a$ has a $C^0$-paraheterocycle at $a=a_0$.

\begin{theo} \label{chain}
Consider a $C^\infty$ family of local diffeomorphisms $(f_a)_{a\in \R}$ in $\Diffloc^\infty(U,M)$
and a heterocycle for $f_0$ between a saddle point $P$ with period $p$ and a projectively hyperbolic source $S$. Let us assume furthermore that the stable eigenvalue of $D_Pf^p_0$ is negative.

Then there exists a family $(\tilde f_a)_{a\in \R}$, $C^\infty$-close to $(f_a)_{a\in \R}$,
which displays a $C^\infty$-paraheterocycle at $a=0$ between the continuation of the saddle $P$
and a projectively hyperbolic source $S'$.
\end{theo}
We will see in \cref{negative-eigenvalue} that the assumption on the negative stable eigenvalue can be obtained when the heterocycle is included in a bicycle.
\begin{remark}
The definition of paraheterocycle, the statement of \cref{chain} and its proof extend without difficulty
to families parametrized by $\R^k$, for any $k\geq 1$, see \cref{r.k-para} and \cref{ss.k-para}.
\end{remark}

\subsection{Parablenders} \label{sec:parablender}
In this section we fix $1\le r<\infty$. 

Parablenders are a parametric counterpart of blenders. 
The first example of a parablender was given in \cite{BE15}; in \cite{BCP16} a new example of parablender was given and therein the definition of parablender was formulated as:
\begin{definition}[$C^r$-Parablender]\label{def.parablender}
The continuation $(K_a)_{a\in I}$ of a hyperbolic set $K$ for the family $(f_a)_{a\in \R}$  is a \emph{$C^r$-parablender} at ${a_0}\in \operatorname{Interior}(I)$ if the following condition is satisfied.

There exist  a continuous family of local unstable manifolds $(W^u_{loc}(\underline z; f _{a_0}))_{\underline z\in\arr K}$ and a non-empty open set  $O$ of germs at $a_0$ of $C^r$-families of points $(\gamma_a)_{a\in I}$ in $M$
such that for every $(\tilde f_a)_{a\in \R}$ $C^{r}$-close to  $(f_a)_{a\in \R}$, there exists $\underline z\in \arr K$
satisfying:
\[\lim_{a\to a_0} |a-a_0|^{-r}\cdot  d\bigg(\gamma_a\;,\; W^u_{loc}(\underline z; \tilde f _a)\bigg)= 0\; .\]

The open set $O$ is called an \emph{activation domain} for the $C^r$-parablender $(K_a)_{a\in I}$. 
\end{definition}

Here is the parametric counterpart of the nearly affine blender introduced in Def.  \ref{exam.blender}.

\begin{definition}[Nearly affine parablender\footnote{ 
The coordinates considered in~\cite{BCP16} were slightly different but the same modulo conjugacy: the renormalized inverses branches are of the form: 
\[B_b^\pm: (X,Y)\mapsto \left(0,   (Y\pm1)/(\Delta^{-1} +b) \right),\]
 which is conjugate to the presented form $(A_a^\pm)_\pm$  via the coordinates changes:
 \[(X,Y)= (x-x_0, \frac{\Delta-1}{\Delta+a}\cdot y)\qand 
a=-\frac{b\cdot \Delta^2 }{1+b\cdot \Delta}\, .\] 
}
 \cite{BCP16}
]\label{parablenders defi}
For $\Delta>1$, $x_0\in (-2,2)$ and $\delta>0$,
a $C^r$-family $(f_a)_{a\in \R}$ has a \emph{$\delta$-nearly affine $C^r$-parablender}
with contraction $\Delta^{-1}$ at $a=0$ if there exist a neighborhood $I$ of $0$ in $\R$, a $C^r$-family $(H_a)_{a\in I}$ of charts  $H_a \colon \R^2\hookrightarrow  M$,
a diffeomorphism $\theta:J\hookrightarrow I$ fixing $0$ and inverse branches $(g^+_{a})_{a\in I}$, $(g^-_{a})_{a\in I}$ of iterates $f_a^{N^+}$, $f_a^{N^-}$ such that 
\[\cR g_{a}^+ := H_a^{-1}\circ g_{\theta(a)}^+\circ H_a 
\qand  \cR g_a^ - := H_a^{-1}\circ g_{\theta(a)}^ -\circ H_a\]
are well defined on $[2,2]^2$ and  $(\cR g_{a}^\pm)_{a\in I} $ are  $\delta$-$C^r$-close to the two families $(A_a^\pm)_{a\in I}$
defined by
\[A_a^+: (x,y)\mapsto \left(x_0,(\Delta+a) \cdot y + \Delta-1 \right)
\qand
A_a^-: (x,y)\mapsto \left(x_0,(\Delta +a)\cdot  y - \Delta+1\right). \]
\end{definition}
Note that a nearly affine parablender defines a germ of family of nearly affine blenders $(K_a)_{a\in I}$ at $a=0$ and so a germ of family of blenders by \cref{ are blender}. In   \cite[Section 6]{BCP16}, we showed\footnote{The activation domain is not explicited in the statements of the results of \cite[Section 6]{BCP16}, but appears in the proof
as a product $W=B\times A$ (see page 67), where $B=[-2,2]\times (-\eta,\eta)^r$ and where
$A$ is a neighborhood of $0$ in $\R^{r+1}$, obtained as the image of a neighborhood of $0$
by a surjective linear map (page 63).} that it defines also a parablender:
\begin{proposition}\label{nearly are parablender} For every $\Delta>1$ close to $1$ and $x_0\in (-2,2)$, there is  $\eta>0$ arbitrarily small such that  if $\delta>0$ is sufficiently small,
then $(K_a)_{a\in I}$ is a $C^r$-parablender at $a=0$. Moreover, its activation domain contains:
\[\left\{(z_a)_{a\in I}\in C^r(I, \R^2) : 
z_0\in [-2,2]\times (-\eta,\eta)  
\qand \| \partial_a^k z_a|_{a=0}\|<\eta   , \quad \forall 1\le k\le r  \right\}.\] 
\end{proposition}
We will show that nearly affine $C^r$-parablenders appear as renormalizations of the dynamics nearby paraheterocycles.  This will enable us to show:
\begin{theo}\label{theosectionparadense} 
Let us consider a $C^\infty$ family of local diffeomorphisms $(f_a)_{a\in \R}$ in $\Diffloc^\infty(U,M)$
and, for $r\geq 1$, a family of saddles $(P_a)_{a\in \R}$ and a family of projectively hyperbolic sources $(S_a)_{a\in \R}$ exhibiting a  $C^r$-paraheterocycle at $a=0$.

Then there exists $(\tilde f_a)_{a\in \R}$, $C^\infty$-close to $(f_a)_{a\in \R}$ displaying a $C^r$-parablender at $a=0$ which is homoclinically related to $P_0$ and whose activation domain contains the germ of $(S_a)_{a\in \R}$ at $a=0$. In particular  $(\tilde f_a)_{a\in \R}$ displays a robust $C^r$-robust paraheterocycle   at $a=0$. 
\end{theo}

\section{Structure of the proofs of the theorems}
  \label{Proof robust hetero}

\subsection{Proof of \cref{robust hetero}}\label{sec:robust robust hetero}
The strategy of the proof breaks down into two steps.
In a first step, we obtain, by perturbation of the heterocycle, a strong heterocycle.
This is done in Section~\ref{s.strong}.

\begin{proposition} \label{pingpong} For $\rho\in \{\infty ,\omega\}$,  let $  f\in \Diffloc^\rho(U,M)$ with a projectively hyperbolic source $S$ and a saddle point $P$ forming a heterocycle. Then there exists a map $\tilde f$ arbitrary $  C^\rho$-close
 with a saddle point $Q$ homoclinically related to $P_{\tilde f}$ and which forms with $S_{\tilde f}$ a strong heterocycle.
\end{proposition}

In a second step we perturb the strong  heterocycle in order to exhibit a nearly affine blender displaying a robust heterocycle. See Section~\ref{s.blender}.

\begin{proposition}\label{PPaffine} 
For $\rho\in \{\infty ,\omega\}$,  let $f\in \Diffloc^\rho(U,M)$ with a projectively hyperbolic source $S$ and a saddle point $Q$ forming a strong heterocycle. Fix $\infty >r\geq 1$ and take $\Lambda>1$ close to $1$.

Then,  for every $\delta>0$ there exists a $C^\rho$-perturbation $\tilde f$ exhibiting a $\delta-C^r$-nearly affine blender which is homoclinically related to $Q_{\tilde f}$ and whose activation domain contains $S_{\tilde f}$.
\end{proposition}
Note that the conjunction of these two propositions implies \cref{robust hetero} for the topologies $C^\infty$ and $C^\omega$.
When the initial diffeomorphism is $C^r$, $1\leq r<\infty$,
we first perturb in the $C^r$-topology in order to get a $C^\infty$-diffeomorphism
taking care that the source $S$ still belongs to the unstable manifold of the saddle $P$, and we then apply the result
for $C^\infty$-diffeomorphisms.
\qed

\subsection{Proof of  \cref{theosectionparadense}}
Similarly to the proof of \cref{robust hetero},
the proof consists in two steps that are the parametric counterparts of \cref{pingpong}  and \cref{PPaffine}.
They are detailed in Sections~\ref{s.strong} and~\ref{s.blender}.

\begin{proposition}\label{Ppingpong}
Consider a $C^\infty$ family of local diffeomorphisms $(f_a)_{a\in \R}$ in $\Diffloc^\infty(U,M)$,
and, for $r\geq 1$, a family of saddles $(P_a)_{a\in \R}$ and a family of projectively hyperbolic sources $(S_a)_{a\in \R}$ exhibiting a  $C^r$-paraheterocycle at $a=0$.

Then there exist $(\tilde f_a)_{a\in \R}$, $C^\infty$-close to $( f_a)_{a\in \R}$
with a family of saddles $(Q_a)_{a\in \R}$ homoclinically related to $(P_a)_{a\in \R}$
which forms with $(S_a)_{a\in \R}$ a strong $C^r$-paraheterocycle at $a=0$. 
\end{proposition} 

\begin{proposition}\label{PPPaffine}
Consider a $C^\infty$ family of local diffeomorphisms $(f_a)_{a\in \R}$ in $\Diffloc^\infty(U,M)$,
and, for $r\geq 1$, a family of saddles $(Q_a)_{a\in \R}$ and a family of projectively hyperbolic sources $(S_a)_{a\in \R}$ exhibiting a strong $C^r$-paraheterocycle at $a=0$.

Then there exists $(\tilde f_a)_{a\in \R}$, $C^\infty$-close to $(f_a)_{a\in \R}$, displaying a $C^r$-parablender at $a=0$ homoclinically related to $Q_0$ and whose activation domain contains the germ of $(S_a)_{a\in \R}$ at $a=0$.
\end{proposition} 
This completes the proof of \cref{theosectionparadense}.   \qed

\begin{remark}\label{r.H4}
One can choose the parablender and the family of local unstable manifolds defining its activation domain
in such a way that each local unstable manifold does not have $S_0$ as an endpoint and is not tangent to the weak unstable direction of  $S_0$. See \cref{H4}.
\end{remark}

\subsection{Proof of \cref{chain}: chains of heterocycles}
We begin with some preparation lemmas. The first one is proved in section~\ref{s.repu}.
\begin{lemma} \label{Repu}   Let $S$ and $P$ be a projectively hyperbolic source and a saddle point forming a heterocycle for a smooth map $f$. Then for a $C^\infty$-small perturbation of the dynamics, the source $S$ belongs to a Cantor set $R$ which is a projectively hyperbolic expanding set. 
\end{lemma}

We introduce the following notion.
\begin{definition}\label{def.Nchain}
A \emph{$N$-chain of alternate heterocycles} for a map $f\in \Diffloc(U,M)$   is the data of 
$N$ saddle points $P^1,\dots,  P^N$ and $N$ projectively hyperbolic sources $S^1, \dots, S^N$ such that:
\begin{itemize}
\item the orbits of $P^1,\dots,P^N,S^1,\dots,S^N$ are pairwise disjoint,
\item the stable eigenvalues of the saddles $P^i$ are negative,
\item $W^u(P^i)$ contains $S^i$ and is transverse to $E^{cu}_{S^i}$ for each $1\leq i\leq N$,
 
\item $W^{uu}(S^i)$ intersects transversally $W^{s}(P^{i+1})$ for $1\leq i<N$
and $W^{uu}(S^N)$ intersects transversally $W^{s}(P^{1})$.
\end{itemize}
\end{definition}

\begin{figure}[h]
    \centering
        \includegraphics{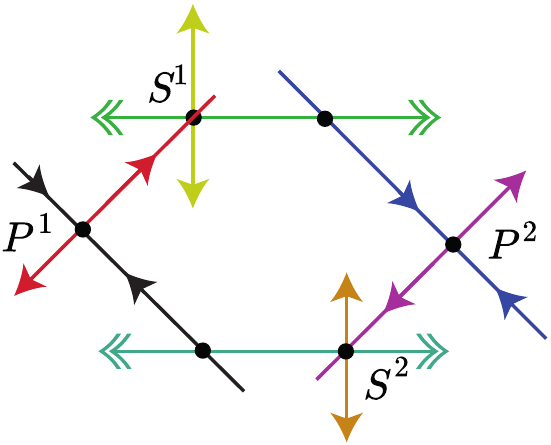}
    \caption{2-Chain of heterocycles.}
    \label{para2}
\end{figure}

Chains of alternate heterocycles may be obtained as follows.

\begin{lemma} \label{L.GenGaraheteroPara}
Consider $f\in \Diffloc^\infty(U,M)$
with a heterocycle between a saddle $P$ with period $p$ and a source $S$ such that the stable eigenvalue of $D_Pf^p$ is negative.

Then, for any $N\geq 1$, there exists $\tilde f$, $C^\infty$-close to $f$,
with an $N$-chain of alternate heterocycles whose
saddles $P^1=P,P^2,\dots,P^N$ are homoclinically related to the continuation $P_{\tilde f}$.
\end{lemma}
\begin{proof}
By preliminary perturbations one stabilizes the heterocycle and builds a blender $K$
homoclinically related to $P$, whose activation domain contains $S$ (\cref{robust hetero}).
One also reduces to the case where the source $S$ belongs to a projectively hyperbolic expanding
invariant Cantor set $R$ (\cref{Repu}).
One can also assume that $W^{uu}(S)$ and $W^s(P)$ have a transverse intersection point.
In order to simplify, one will assume that $K$ is topologically mixing
(otherwise one has to decompose $K$ into finitely many pieces permuted by the dynamics
and whose return map is topologically mixing on each piece).
Note that $P^1=P$ and $S^1=S$ define a $1$-chain of alternate heterocycle.
One proves the statement by induction on $N$.
Let us assume that $f$ has a $N-1$-chain of alternate heterocycles whose
saddles $P^i$ are homoclinically related to $P$.

One chooses a saddle $P^N$ whose orbit is distinct from the orbits of
$P^1,\dots,P^{N-1}$ and which is homoclinically related to $P$: since $W^{uu}(S^{N-1})$ intersects transversally $W^s(P^{1})$, it also intersects transversally $W^s(P^{N})$.
One also chooses a source $S^N\in R$ in the activation domain of $K$
and whose orbit is distinct from the orbits of $S^1,\dots,S^{N-1}$; one can furthermore assume that it is arbitrarily close to $S$, so that $W^{uu}(S^N)$ intersects transversally $W^s(P)$, hence $W^s(P^1)$.
The blender property implies that $S^N$ belongs to the unstable set of $K$.
More precisely there exists $x\in K$ and $y\in W^u(x)\setminus \text{\rm Orbit}(S^N)$
such that $f(y)\in \text{\rm Orbit}(S^N)$.
Since $K$ is topologically mixing, $W^{u}(P^N)$ is dense in the unstable set of $K$,
one can find $y'\in W^u(P^N)$ arbitrarily close to $y$
and whose backward orbit is disjoint from a uniform neighborhood of $y$.
One then perturbs $f$ in a small neighborhood of $y$
and get a map satisfying $\tilde f(y')=f(y)$.
Consequently $P^N$ and $S^N$ define a heterocycle for $\tilde f$
and the properties built at the previous steps of the induction are preserved.
\end{proof}

The existence of a saddle point with negative stable eigenvalue
may be obtained once a saddle belongs to a homocycle, as we recall in the next lemma.
\begin{lemma} \label{negative-eigenvalue}
Let $f\in \Diffloc^\infty(U,M)$ and $P$ be a saddle point with a homoclinic tangency $L$. Then for a $C^\infty$-small perturbation $\tilde f$ of the dynamics supported on a small neighborhood of $L$,
the saddle $P$ belongs to a basic set which contains a point $Q$ with some period $\tau$
and such that the stable eigenvalue of $D_Q\tilde f^\tau$ is negative.
\end{lemma}
\begin{proof}
This is a well-known result.
Up to replace $L$ by an iterate, one assumes $L\in W^s_{loc}(P)$.
One perturbs $f$ so that the contact of the homoclinic tangency is quadratic.
By unfolding the homoclinic tangency, a horseshoe containing $P$ appears.
Indeed, one considers a thin rectangle $R$ which is a tubular neighborhood of $W^s_{loc}(P)$.
A large iterate $f^\ell(R)$ crosses $R$ twice, with different orientations.
In each component of the intersection, a $\ell$-periodic point is obtained,
and the signs of $Df^\ell$ along the stable direction differ.
See \cite[chapter 3]{PT93} for details.
\end{proof}

\cref{chain} follows from the next proposition, proved in \cref{s.heterocycles}.

\begin{proposition} \label{P.GenGaraheteroPara}
For any $d\geq 0$, there exists $N=N(d)\geq 1$ with the following property.

Consider a $C^\infty$ family
$(f_a)_{a\in \R}$  in $\Diffloc^\infty(U,M)$
such that $f_0$ has a $N$-chain of alternate heterocyles with saddle points $P^i$ and sources $S^i$.
Then there exists a family $(\tilde f_a)_{a\in \R}$, $C^\infty$-close to $(f_a)_{a\in \R}$,
such that the continuations of $P^1$ and $S^{N}$  form a $C^d$-paraheterocycle at $a=0$.
\end{proposition}

\begin{remark}\label{r.k-para}
This result is still valid for families parametrized by $\R^k$, $k\geq 1$ (see Section~\ref{ss.k-para}).
The length of the chain required is then equal to:
\[N(r,k)={\dim_\R \{P\in \R[X_1, \dots, X_k]:\; \deg P\le r,\; P(0)=0\}}.\] 
\end{remark}

\begin{proof}[Proof of \cref{chain}]
For any large integer $d\geq 1$, \cref{L.GenGaraheteroPara} and
\cref{P.GenGaraheteroPara} give after a $C^\infty$-perturbation
a $C^d$-paraheterocycle between the continuation of the saddle $P$
and a projectively hyperbolic source $S'$.
Hence there exists a $C^d$-small perturbation $(f'_a)_{a\in \R}$
in $\Diffloc^\infty(U,M)$ and an integer $N$ which satisfy
$S'_a\in f^N(W^u_{loc}(P_{f'_a}))$ for any $a$ close to $0$.
Since $d$ has been chosen arbitrarily large, the perturbation can
be chosen $C^\infty$-small.
\end{proof}

\subsection{Proof of \cref{main}}
A consequence of the previous results is the:
\begin{coro}\label{thesis H}
Consider a $C^\infty$ family $(f_a)_{a\in \R}$ in $\Diffloc^\infty (U,M)$ such that $f_0$ displays a bicycle between a projectively hyperbolic source $S_0$ and a dissipative saddle point $P_0$. Let $r\geq 1$. Then up to a $C^r$-perturbation of the family, and up to replacing $S_0$
by another projectively hyperbolic source, we can assume that:
\begin{enumerate}[$(H_1)$]\setcounter{enumi}{-1}
\item There exists a blender $K_0$ for $f_0$ whose activation domain contains $S_0$.
\item $K_0$ intersects the repulsive basin of $S_0$.
\item $P_0$ is homoclinically related to $K_0$ and $W^s (P_0) $ has a robust tangency with the strong unstable foliation $\cF^{uu}$ of $S_0$.
\item The continuation $(K_a)_{a\in I}$ of $K_0$ is a $C^r$-parablender at $a=0$ and the continuation  $(S_a)_{a\in I}$ of $S_0$ belongs to its activation domain.
\item  In the continuous family of local unstable manifolds defining the activation domain involved in $(H_0)$ and $(H_3)$,
each local unstable manifold does not have $S_0$ as an endpoint and is not tangent to the weak unstable direction of  $S_0$.
\end{enumerate}
\end{coro}
\begin{remark}
The properties $(H_0)$\dots$(H_4)$ are $C^r$-open.
\end{remark}
\begin{proof}
With \cref{robust-bicycle} \cpageref{robust-bicycle}, one first stabilizes the bicycle.
By \cref{negative-eigenvalue}, up to a small $C^\infty$-perturbation, one gets a saddle $Q_0$
homoclinically related to $P_0$ whose stable eigenvalue at the period is negative.
One thus gets a robust heterocycle between $Q_0$ and $S_0$
and  \cref{chain} \cpageref{chain} gives a family $(f'_a)_{a\in \R}$, that is $C^\infty$-close, and displaying a $C^\infty$-paraheterocycle between the continuation of $Q_0$ and a projectively hyperbolic source saddles $S'$.
\cref{theosectionparadense} \cpageref{theosectionparadense} produces a family $C^r$-close having a
$C^r$-parablender $(K_a)_{a\in I}$ at $a=0$ which is homoclinically related to $Q_0$ (and $P_0$)
and whose activation domain contains the family of source $(S'_a)_{a\in I}$.
Denoting the new source by $S_0$, we get all the robust properties $(H_0)$, $(H_1)$ and $(H_3)$.
By \cref{r.H4}, $(H_4)$ is also satisfied.

Since $P_0$ and $S_0$ form a robust heterocycle,
one can assume (after a new perturbation) that the strong unstable manifold $W^{uu}(S_0)$
intersects transversally $W^s(P_0)$.
From the robust tangency, we can perturb and produce a homoclinic tangency point $L$
between $W^u_{loc}(P_0)$ and $W^s(P_0)$.
The inclination lemma implies that $W^{uu}(S_0)$ accumulates on $W^u_{loc}(P_0)$.
A last perturbation near $L$ gives a quadratic tangency between $W^{uu}(S_0)$
and $W^s(P_0)$. For maps close, this tangency admits a continuation
which is a quadratic tangency between $W^s(P_0)$
and the leaves of the strong unstable foliation in the repelling basin of $S_0$:
this is $(H_2)$.
\end{proof}

\begin{figure}[h!]
\begin{center}
\includegraphics[width=6cm]{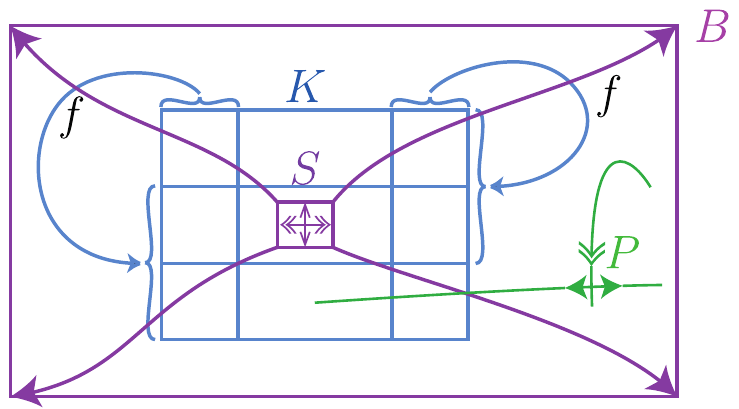}
\caption{Assumptions $(H_0)$\dots$(H_4)$.  \label{fig:assumption H}}
\end{center}
\end{figure}

We now use the following result of \cite[Theorem A, page 11]{berger2017emergence}:
\begin{theorem}\label{t.infinite-sinks}
Consider a $C^\infty$ family $(f_a)_{a\in \R}$ in $\Diffloc^\infty (U,M)$
with a projectively hyperbolic source $(S_a)_{a\in \R}$ and a dissipative saddle point $(P_a)_{a\in \R}$ satisfying $(H_0)$\dots$(H_4)$.
Then,  there are  $\delta>0$,  a $C^r$-neighborhood $\mathcal V$
of the family  $(f_a)_a$
in the space of $C^r$-families
and a Baire-generic subset $\cal G\subset \cal V$ such that for any $(\tilde f_a)_a\in \cal G$ and  $a\in (-\delta, \delta)$, the map $\tilde f_a$ displays infinitely many sinks.
\end{theorem}
For completeness we sketch its proof.
\begin{proof}[Idea of the proof of \cref{t.infinite-sinks}] 
Since the hypothesis are open, they hold for an open neighborhood $\mathcal V$ of the initial family.
Let us consider an arbitrary family $(f'_a)_{a\in \R}$ in $\mathcal V$.
The robust heterocycle provided by $(H_0)$ and $(H_1)$ and \cref{Repu} allow after a perturbation to assume that there are $\delta>0$
and two distinct sources $(S_a)_{a\in [-\delta, \delta]}$, $(S_a')_{a\in [-\delta, \delta]}$ which satisfy $(H_0)$\dots$(H_4)$ at every $a_0\in [-\delta, \delta]$ for each of these sources
and for the family $(f'_a)_{a\in \R}$.

Then we apply the following key lemma (which uses $(H_3)$):
\begin{lemma}[{\cite[Prop. 3.6]{berger2017emergence}}]\label{prop36}
For every $\varepsilon>0$, there exist $\alpha>0$ and an $\varepsilon$-$C^r$-perturbation $(f''_a)_{a\in [-\delta, \delta]} $ localized at  $(S_a)_a$ and $(S_a')_a$ such that:
\begin{enumerate}
\item  for every $j\in 2\Z$, there exists a continuation of a periodic point $(P^{(j)}_a)_a$ in the parablender whose local unstable manifold contains $S_a$ for every $a\in [-\delta, \delta] \cap [\alpha j-\alpha/2,\alpha j+\alpha/2]$,  
\item for every $j\in 2\Z+1$, there exists a continuation of a periodic point $(P^{(j)}_a)_a$ in the parablender whose local unstable manifold contains $S'_a$ for every $a\in [-\delta, \delta] \cap [\alpha j-\alpha/2,\alpha j+\alpha/2]$.
\end{enumerate}
\end{lemma}
We continue with:
\begin{lemma}[{\cite[Prop. 3.4]{berger2017emergence}}] \label{prop34}
After a new $C^\infty$-small perturbation of $(f''_a)_a$,   for every $j\in \Z\cap [-\delta/\alpha, \delta/\alpha]$ the point $P_a$ displays a quadratic homoclinic tangency which persists for every $a\in [-\delta, \delta] \cap [\alpha j-\alpha/2,\alpha j+\alpha/2]$. \end{lemma}
\begin{proof}[Idea of proof of \cref{prop34}]
Assume $j$ odd (resp. even) and let us continue with the setting of \cref{prop36}. 
As $P_a$ and $P^{(j)}_a$ belong to the same transitive hyperbolic set and using \cref{continuation},
after a small perturbation a fixed iterate of the local unstable manifold of $P_a$ contains $S_a$
for every $a\in [-\delta, \delta] \cap [\alpha j-\alpha/2,\alpha j+\alpha/2]$.
Then we proceed  as depicted in \cref{fig:proofofprop34}:  we denote by $W^s_a$ a segment of $W^s(P_a)$  which is included in a basin of $S_a$ (resp. $S_a'$) and display a tangency with the strong unstable foliation of the
repelling basin of $S'_a$ (resp. $S_a$) by $(H_2)$. After perturbation we can assume this tangency quadratic. Then, in the Grassmanian   bundle $\mathbb P(TM)$ of $M$, the tangent space $TW^s$ of this curve intersects transversally the unstable manifold of $(S_a, E^{uu}(S_a)) $ for the action $\hat f_a$ of $Df_a$ on the Grassmannian. By the inclination lemma, the  preimages $TW^s_{n,a}$  of $TW^s_a$, by $\hat f_a^n$, converge to   the stable manifold $\{S_a\}\times \mathbb {P R}^1\setminus \{E^{cu}(S_a)\}$ of $(S_a, E^{uu}(S_a)) $.
By property $(H_4)$, a piece of $W^u_{loc} (P_a)$ intersects $S_a$ with a direction different from $E^{uu}(S_a)$,
hence the stable manifold of $(S_a, E^{uu}(S_a))$ intersects untangentially a piece $TW^u_a$ of $TW^u_{loc} (P_a)$ for every $a\in [-\delta, \delta] \cap [\alpha j-\alpha/2,\alpha j+\alpha/2]$. This enables to perturb $(f_a)_a$ such that $TW^s_{n,a}\subset TW^s (P_a)$ intersects $TW^s_{loc} (P_a)$ for every $a\in [-\delta, \delta] \cap [\alpha j-\alpha/2,\alpha j+\alpha/2]$. \begin{figure}[h!]
\begin{center}
\includegraphics[width=7cm]{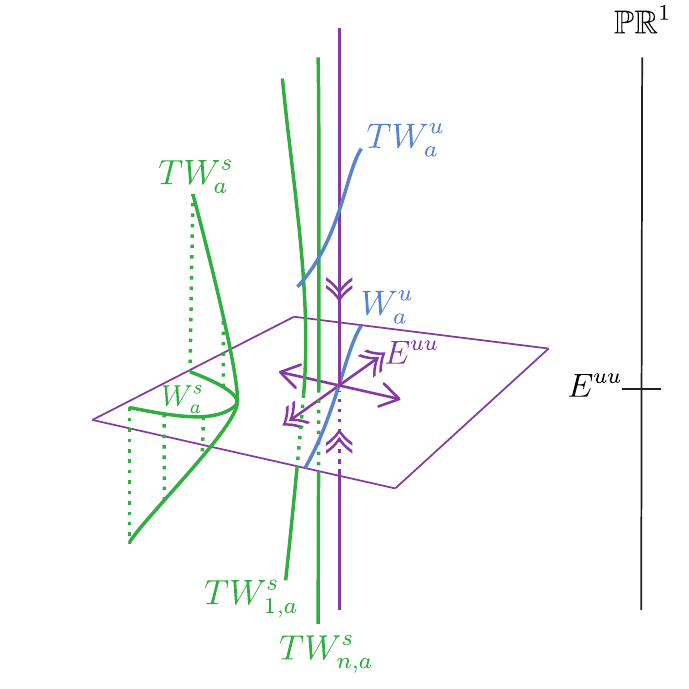}
\caption{Inclination lemma used in the bundle $\mathbb P(TM)$ at $E^{uu}(S_A)$. }\label{fig:proofofprop34}
\end{center}
\end{figure}
\end{proof}
In {\cite[Prop. 3.5]{berger2017emergence}} it is shown that for every $N\geq 1$, we can then perturb the family
in the $C^\infty$-topology near the homoclinic tangency of $(P_a)_a$ obtained in \cref{prop34}
so that for every $a\in  [\alpha j-\alpha/2,\alpha j+\alpha/2]$ and $j\in  \Z\cap [-\delta/\alpha, \delta/\alpha]$, the new map $\tilde f_a$ displays a periodic sink of period $\ge N$.  Hence we have obtained an open and dense subset in $\mathcal V$ of families displaying a sink of period $\ge N$ at every parameter $a\in [-\delta, \delta]$.  By taking the intersection $\cal G$ of these open and dense subsets over $N\ge 1$, we obtain \cref{t.infinite-sinks}. 
 \end{proof}

This allows to complete the proof of our main theorem.
\begin{proof}[Proof of \cref{main}]
Let us consider a $C^r$ map $f$ with a dissipative bicycle associated to a saddle $P$.
By \cref{{robust-bicycle}}, there exists a $C^r$-open set $\cU\in \Diffloc^r(U,M)$, which contains $f$ in its closure,
such that the continuation of $P$ exhibits a robust bicycle for any map in $\cU$.

Let $F:=(f_a)_{a\in \R}$ be a $C^r$-family consisting of maps $f_a\in\cU$.
By perturbation, one can assume that the family is $C^\infty$ and
by \cref{densitebicycle} that $f_0$ displays a bicycle.
Then, by \cref{thesis H}, there exists a new $C^r$-perturbation which
satisfies 
$(H_0)\cdots(H_4)$.
\cref{t.infinite-sinks} associates a neighborhood $\mathcal{V}_F$ of this family and a dense 
G$_\delta$-set ${\mathcal G}_F$ of $\mathcal{V}_F$ and $\delta_F>0$.
Let $\{F_n: n\in\N\}$ be a dense countable set in the space of families $(f_a)_{a\in \R}\in \mathcal{D}^r(\I \times U,M)$
consisting of maps $f_a\in\cU$.
The union ${\mathcal G}=\bigcup {\mathcal G}_{F_n}$ is a dense $G_{\delta}$ subset of this space.
By construction, for any family $F=(f_a)_{a\in \R}$ in ${\mathcal G}$
and any $|a|$ smaller than a semi-countinuous function $\delta_F$ of $F$, the map $\tilde f_a$ exhibits infinitely many sinks for any parameter $a$ close to $0$.
\end{proof}

 \section{From heterocycles to basic sets and strong heterocycles}\label{s.strong}
In this section we prove \cref{pingpong}, \cref{Ppingpong} and \cref{Repu}.

We consider a $C^\infty$ map $f\in \Diffloc^\infty(U,M)$ with a projectively hyperbolic source $S$ and a saddle point $P$ forming a heterocycle, and we show that by perturbation it can be improved to
a strong heterocycle.

In \cref{Local coordinate for a heterocycle}, first we establish local coordinates around $P$ and $S$. To obtain these coordinates, we need to perturb the dynamics, to assume the eigenvalues non-resonant, but also to ensure two transversality conditions $(T_1)$-$(T_2)$. Then nearby  $P$ and $S$, the inverse dynamics $\cP$ and $\cS$ are linear in local coordinates. Furthermore, the heterocycle defines inverse branches of the dynamics that are transitions from one linearizing chart to the other.    

As a direct application of these linearizing charts, we build an IFS
and from there an expanding projective hyperbolic set containing the source: this allows to prove \cref{Repu} at the beginning of \cref{section basic}).
Later, using again  these coordinates, we obtain the existence of a non-trivial basic set $K$ which contains $P$
(\cref{existence of Q}).
After a small perturbation, which consists in perturbing the stable eigenvalues of $P$, the strong unstable manifold of $S$ intersects $K$, whereas $S$ belongs to $W^u(K)$. This will imply   \cref{pingpong}. The proof of \cref{Ppingpong} follows similar lines.

\subsection{Local coordinates for a heterocycle}\label{Local coordinate for a heterocycle}
For the sake of simplicity we assume that the periodic points $P$ and $S$ are fixed
and that the eigenvalues of $D_Pf$ and $D_Sf$ are positive.
Anyway we can go back to this case by regarding an iterate of the dynamics and performing the forthcoming perturbations
nearby finitely many points belonging to different orbits.

Up to a smooth perturbation we can assume that the eigenvalues of $D_Pf$ and $D_Sf$ are non-resonant.
Then Sternberg Theorem \cite{S58} implies the existence of:
\begin{itemize} 
\item  neighborhoods $V'_S\subset V_S:= f(V'_S)$ of $S$ and coordinates for which $f|V'_S$ has the form:
\[ (x,y)\in V'_S\mapsto (\sigma_{uu}^{-1}\cdot  x, \sigma_u^{-1}\cdot  y)\in V_S \quad  \text{with } 0< \sigma_{uu} <   \sigma_u <1\; 
.\]
\item    neighborhoods $V'_P$ and $V_P:= f(V_P')$  of $P$ and  coordinates for which  $f|V'_P$ has the form:
\[ (x,y)\in V'_P\mapsto (\sigma^{-1}\cdot  x, \lambda^{-1}\cdot  y)\in V_P\quad  \text{ with } 0<  \sigma  <1<   \lambda  \; 
.\]
\end{itemize}
This defines the inverse branches $\cP:= (f|V_P')^{-1}$ and 
 $\cS:= (f|V_S')^{-1}$:
\[\cS:  (x,y)\in V_S\mapsto (\sigma_{uu}  \cdot  x, \sigma_u \cdot  y)\in V_S'\qand \cP:  (x,y)\in V_Q\mapsto  (\sigma \cdot  x, \lambda \cdot  y)\in  V_Q'\; .\]
Up to restricting  $V_P$ and $V_P'$ and rescaling the coordinates, we can assume:
\[V_P'\equiv [-\sigma,\sigma]\times [-1,1]\qand 
V_P\equiv [-1,1]\times [-\lambda^{-1},\lambda^{-1}]
\; .\]
Let $W^u_{loc}   (P):= V_Q\cap \{y=0\} $,   $W^s_{loc}   (P):= V'_Q\cap \{x=0\} $ and $W^{uu}_{loc}(S):= \{y=0\} \cap V_S$.

Let $H$ be a point in $ W^s(P)\cap W^{uu}(S)$. Up to replacing it by an iterate, we can assume that $H$ belongs to $V'_P$ with $H=:(0,h)$  in the linearizing coordinates of $P$.  Up to 
conjugating the dynamics by $(x,y)\mapsto (x,-y)$, we can assume moreover that $h>0$. 
Also, a preimage $S'$ of $S$  has coordinates $S' =:(s,0)$ in  the linearizing coordinates of $P$:
\[S'\equiv (s,0)\qand   H\equiv (0,h)\;, \quad h>0\; .\]
Furthermore up to a smooth perturbation, we can assume that:
\begin{enumerate}[$(T_1)$]
\item\label{T1} The intersection $ W^s(P)\cap W^{uu}(S)$ is transverse at $H$. 
\item\label{T2} The line $ T_S W^u(P)$ is in direct sum with the weak unstable direction $E^{cu}$ of $S$. 
\item The line $ T_S W^u(P)$ is in direct sum with  the strong unstable direction $E^{uu}$ of $S$. 
\end{enumerate} 

Let $V''_S \Subset V'_S$ and $V_H \Subset V'_P$ be small neighborhoods of $S$ and $H$;
and let $\TS: V_S''\hookrightarrow V_P$ and $\TH: V_H\hookrightarrow V_S$ be inverse branches
of iterates of $f$ such that $\TS(S)=S'$ and $\TH(H)\in W^{uu}_{loc}(S)$.

\begin{figure}[h!]
\begin{center}
\includegraphics[width=11cm]{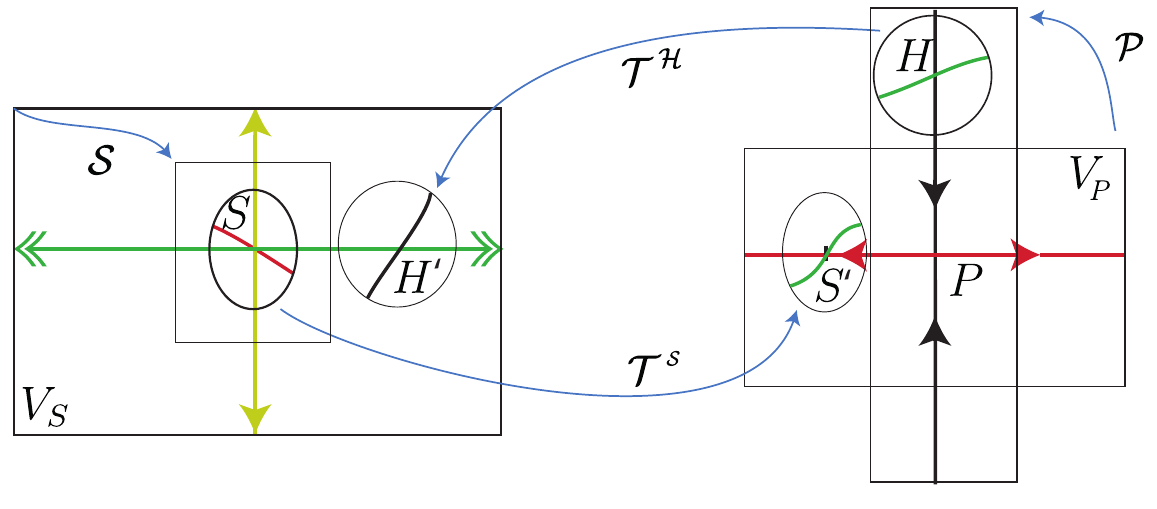}
\caption{Inverse branches $\cS, \cP, \TS, \TH$ induced by the  heterocycle.  \label{fig:def_inversebranches}}
\end{center}
\end{figure}

\subsection{Basic sets induced by a heterocycle}\label{section basic}
We now build two hyperbolic sets: one expanding projective hyperbolic set containing the source, and a saddle hyperbolic set containing the saddle.

\subsubsection{Proof of \cref{Repu}: expanding Cantor set linked to the heterocycle} \label{s.repu}

Note that for $n$ large, the point $(s, \lambda^{-n} h)$ belongs to the range of $\TS$.
We perturb $f$ near the point $S'$ and define a map $\tilde f$
which satisfies in coordinates \[\tilde f(s, \lambda^{-n} h)=f(S').\]
This in turn defines a perturbation $\widetilde \TS$ of the inverse branch $\TS$.

As the point $(s, \lambda^{-n} h)$ is sent by $\cP^n$ to the $(\sigma^n \cdot s,  h)\in V_H$, the map $\TH\circ \cP^n \circ \widetilde  \TS$ is well defined on a neighborhood $W$ of $S$. Hence for $N$ large compared to $n$, the maps
$\cS_1:=\cS^N\circ \TH\circ \cP^n \circ \widetilde \TS$
and $\cS_2:= \cS^N$ are contractions from $W$ into $ W$ with disjoint images.
So they define a transitive expanding Cantor set $R$ for $\tilde f$ which contains $S$.

Let us fix $\eta>0$ small and introduce the cone
$\mathcal{C}:=\{(u,v):\; |u|<\eta |v|\}$.
Using (T$_1$), (T$_2$) and assuming that $n,N$ have been chosen large enough,
for any $x\in W$, the maps $D_x\cS_1$ and $D_x\cS_2$ send $\overline {\mathcal{C}}$
inside $\mathcal{C}\cup\{0\}$.
The cone field criterion (see for instance~\cite{yoccoz-hyperbolic}) implies that the Cantor set $R$ is projectively hyperbolic. The \cref{Repu} is proved.
\qed

\subsubsection{Basic sets  linked to the heterocycle}\label{Horseshoes  linked to the heterocycle}
The heterocycle configuration implies under the transversality assumptions $(T_1)$ and $(T_2)$
that the saddle $P$ has a transverse homoclinic intersection.
 \begin{lemma} \label{existence of Q}
For all $n$ large, the subsegment:
 \[W^s_{loc} (\bar H):=\TS\circ \cS^n\circ \TH( W^s_{loc}   (P) \cap V_H)\]
  of $W^s  (P)$ intersects transversally the local unstable manifold $W^u_{loc}   (P)$ at a   point $\bar H$   which is $\asymp \sigma_{uu}^n$-close to $S'$. The endpoints of $W^s_{loc} (\bar H)$ are $\asymp \sigma_u^n$-distant from $W^u_{loc} (P)$. 
\end{lemma}
\begin{proof} 
Let $\Gamma:= W^s_{loc}   (P)\cap V_H$. This curve is sent by $\TH$ to a curve which intersects transversally $W^{uu}_{loc}(S)$ by $(T_1)$.  By projective hyperbolicity, the image by $\cS^n$ of $\TH(\Gamma)$ is a  curve
which is tangent to a thin vertical cone field, which is $\asymp\sigma_{uu}^n$-close to $S$ and which has
length $\asymp  \sigma_u ^n$. As $(\TS)^{-1}(\{y=0\})$ intersects transversally 
$\{x=0\}\cap V_S$ at $S$ by $(T_2)$, it must intersect transversally  $\cS^n\circ \TH(\Gamma)$ for $n$ large. Consequently   the curve  $\TS\circ \cS^n\circ \TH(\Gamma)$  intersects the local unstable manifold $\{y=0\}\cap V_P$ of $P$.
\end{proof}
By Smale's horseshoe theorem (see~\cite[chapter 2]{PT93}), one deduces:

\begin{coro}\label{c.basic}
There exists a basic set $K$ containing $P$ and $\bar H$.
\end{coro}

We will make it more precise.
If $N$ is large, $K$ can be
spanned by the inverse branches
\[\mathcal{G}_1:=\cP^N \quad\text{and}\quad \mathcal{G}_2:=  \TS\circ \cS^n\circ \TH\circ \cP^N.\]

\begin{figure}[h!]
\begin{center}
\includegraphics[width=11cm]{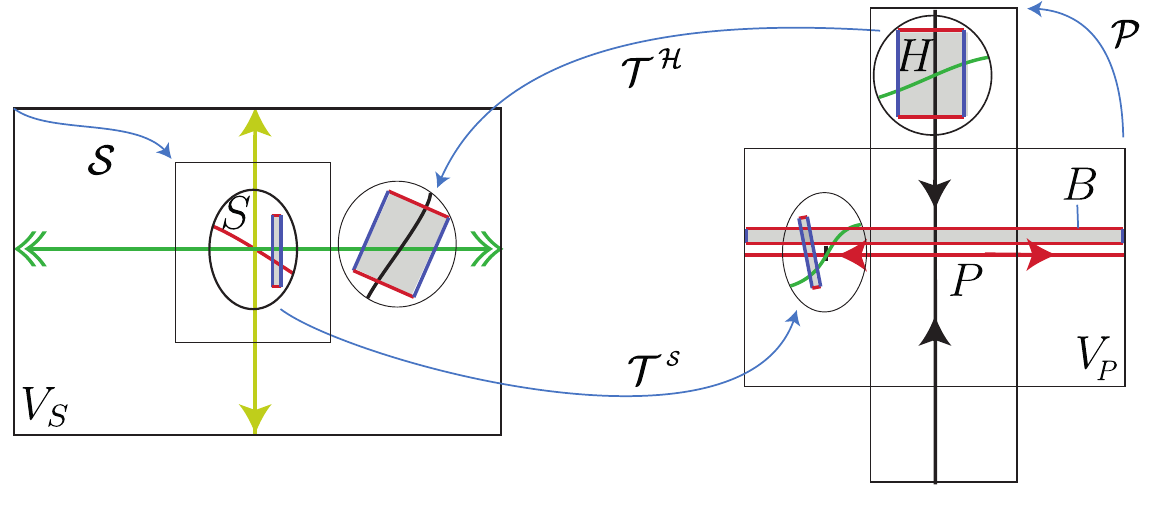}
\caption{The box $B$ and its images.
\label{fig:def_inversebranches_closing}}
\end{center}
\end{figure}

Let $\varepsilon>0$ be small enough so that $\{0\}\times [h-\varepsilon, h+\varepsilon]$ is included in $V_H$ and let
(see  \cref{fig:def_inversebranches_closing}):
 \[ B:=  [-1,1]\times \left[\frac{ h-\varepsilon}{ \lambda^{N}},\frac{ h+\varepsilon}{ \lambda^{N}}\right] 
 \; .\]
\begin{lemma}\label{condition eigen B} For every $n, N$ large, the map $\mathcal{G}_2$ is well defined on $B$. If $\varepsilon\sigma^n_u\lambda^{N}\gg 1$,  the map $\mathcal{G}_2$ displays a saddle fixed point $\bar Q$ in $B\cap K$, which is homoclinically related to $P$.
\end{lemma}
\begin{proof}  
The box $B$ is sent by $\cP^N$ to $ (0,h)+ [-\sigma^N,\sigma^N]\times [-\varepsilon,\varepsilon]$ which is included in $V_H$ for $N$ large enough. As $\cS^n\circ \TH (V_H)$ is included in $V_S''$ for $n$ large enough,  the map $g$ is well defined on $B$.
Let us decompose the boundary of $B$:
\[\partial^s B:=  \{-1,1\}\times \left[\frac{ h-\varepsilon}{ \lambda^{N}},\frac{ h+\varepsilon}{ \lambda^{N}}\right]
\qand \partial^u B= \partial B\setminus \partial^s B.\]
Both curves of $\cS^n\circ \TH\circ \cP^N(\partial^s B)$ are
$\sigma_{uu}^n$ close to the vertical arc $W^c_{loc} (S):=\{0\}\times [-1,1]$
and their endpoints are  $\asymp\varepsilon\cdot \sigma_u^n$ distant to $W^{uu}_{loc} (S)$ by transversality
$(T_1)$ at $\TH(H)$ and by projective hyperbolicity of $S$.  Thus they intersect transversally  $(\TS)^{-1}(W^u _{loc} (P))$ by property $(T_2)$. 

Consequently $\mathcal{G}_2(B)$ intersects 
$W^u _{loc} (P)$, and $\mathcal{G}_2(\partial^u B)$ is $\asymp\varepsilon\cdot \sigma_u^n$ distant to $W^u _{loc} (P)$.
By assumption, $\lambda^{-N}$ is small compared to $\varepsilon \cdot \sigma^n_u$, then
$\mathcal{G}_2(B)$ crosses $B$: it does not meet the vertical boundary $\partial^s B$,
whereas $B$ does not meet the horizontal boundary $\mathcal{G}_2(\partial^u B)$.
Thus $\mathcal{G}_2$ displays a fixed point $\bar Q$ in $B\cap \mathcal{G}_2(B)$. 

Note that $D \mathcal{G}_2$ expands vectors in a vertical cone by a factor $\asymp \lambda^N\cdot \sigma_u^n$, which is large, and the image of these vectors are uniformly transverse to the horizontal. On the other hand by projective hyperbolicity $D \mathcal{G}_2^{-1}$  sends the vectors in an horizontal cone to uniformly horizontal vectors and expands them by a factor $\sigma^{-n}_{uu} \cdot \sigma^{-N}$.
The point $\bar Q$ is a saddle, its local unstable manifold is an horizontal graph in $B$ over $[-1,1]$
whereas its local stable manifold connects the two curves in $\mathcal{G}_2(\partial^u B)$
and so crosses the horizontal $W^u_{loc}(P)$.
This shows that $\bar Q$ and $P$ are homoclinically related as required.
\end{proof}
 
\subsubsection{Replacement of the saddle point}
\label{replacement}
Let us consider a saddle periodic point $Q$ homoclinically related to $P$.
The following allows to replace the saddle $P$ by $Q$ in the heterocycle.

\begin{lemma}\label{l.replace}
Let $Q$ be a periodic saddle point that is homoclinically related to $P$.
Then there exists a map $\tilde f$ that is $C^\infty$ close to $f$
such that $S$ and $Q$ form a heterocycle.

One can choose $\tilde f$ to coincide with $f$ outside an arbitrarily small neighborhood of $f^{-1}(S)\setminus \{S\}$. In particular if $W^{uu}(S;f)$ contains $Q$, then $S$ and $Q$ form a strong heterocycle for $\tilde f$.
\end{lemma}
\begin{proof}
By assumption, there exists a point $z\in W^{u}(P)\cap f^{-1}(S)\setminus \{S\}$.
Let $Q_{-1}$ be the forward iterate of $Q$ which satisfies $f(Q_{-1})=Q$.
Since $Q_{-1}$ is homoclinically related to $P$, there exists $z'\in W^{u}(Q_{-1})$
arbitrarily close to $z$ having a backward orbit which converges to the orbit of $Q$
and which avoids a uniform neighborhood of $z$.

Hence, there exists a $C^\infty$-small perturbation of $f$ supported on a small neighborhood of $z$
satisfying $\tilde f(z')=f(z)$. In particular $W^u(Q)$ contains $S$.
\end{proof}

We state a parametric version of the previous lemma.

\begin{lemma}\label{l.replace-para}
Consider a $C^\infty$ family $(f_a)_{a\in \R}$ in $\Diffloc^\infty(U,M)$,
and, for $r\geq 1$, families of saddles $(P_a)_{a\in \R}$ and of projectively hyperbolic sources $(S_a)_{a\in \R}$ exhibiting a  $C^r$-paraheterocycle at $a=0$.
If $(Q_a)_{a\in \R}$ is a family of saddles homoclinically related to $(P_a)_{a\in \R}$,
then there exists $(\tilde f_a)_{a\in \R}$, $C^\infty$-close to $( f_a)_{a\in \R}$
such that $Q_0$ and $S_0$ form a $C^r$-paraheterocycle at $a=0$. 

One can choose $(\tilde f_a)_{a\in \R}$ to coincide with $( f_a)_{a\in \R}$ outside an arbitrarily small neighborhood of $f^{-1}_0(S_0)\setminus \{S_0\}$. Hence if $Q_0\in W^{uu}(S_0;f_0)$, then $S_0$, $Q_0$ form a strong heterocycle for $\tilde f_0$.
\end{lemma}
\begin{proof}
Let $(K_a)_{a\in I}$ be a basic set that contains $P_a$ and $Q_a$ for $a$ in a neighborhood $I$ of $0$.
Let $\underline P$ and $\underline Q$ be the periodic lifts of $P$ and $Q$ in $\overleftarrow K$.
By assumption, there exists a choice of local unstable manifolds $W^u_{loc}(\underline z,f_a)$
for $\underline z\in \overleftarrow K$ and $N\geq 1$ such that
$d(S_a,f^N(W^{u}_{loc}(\underline P_a)))=o(|a|^r)$.
Since $P$ and $Q$ are homoclinically related, there exists a sequence
of points $\underline z_n\in \overleftarrow K$ which converges to $\underline P$
and which belong to $W^u_{loc}(\underline Q)$.
Since $W^u_{loc} (\underline z; f_a) $ varies continuously with $\underline z$ for the $C^\infty$-topology,
when $n$ is large there exists a family $(\tilde f_a)_{a\in \R}$, which is $C^\infty$-close to $( f_a)_{a\in \R}$,
such that $d(S_a,\tilde f^N_a(W^{u}_{loc}(\underline z_{n},\tilde f_a)))=o(|a|^r)$.
There exists a large integer $\ell\geq 1$ such that
$\tilde f^N_a(W^{u}_{loc}(\underline z_{n},\tilde f_a))\subset \tilde f^\ell_a(W^{u}_{loc}(\underline Q_a,\tilde f_a))$,
hence $d(S_a,\tilde f^\ell_a(W^{u}_{loc}(\underline Q_a,\tilde f_a)))=o(|a|^r)$ as in the definition of $C^r$-paracycle.
Note that the perturbation can be supported
on a neighborhood of a point in $f_0^{-1}(S_0)\setminus \{S_0\}$.
\end{proof}

\subsection{Proof of \cref{pingpong}: from heterocycles to strong heterocycles}
\label{proofpingpong}
The main step in the proof of \cref{pingpong} is contained in the following lemma.
\begin{lemma}\label{l.pingpong}
Let us assume that both stable branches of $P$ intersect $W^u(P)$ transversally.
Then there exists a map $\tilde f$, $C^\infty$-close to $f$,
with a saddle $Q$ homoclinically related to $P_{\tilde f}$
such that:
\begin{itemize}
\item $f$ and $\tilde f$ coincide on $W^u_{loc}(P)$ and outside a small neighborhood of $P$,
\item $W^{uu}(S,\tilde f)$ contains $Q$.
\end{itemize}
\end{lemma}
\begin{proof}
From $(T_1)$,
the curves $W^{uu}_{loc}(S)$ and $\cS^n\circ \TH( W^s_{loc}   (P) \cap V_H)$ intersect transversally
at a point whose image under $\TS$ is denoted as $[S',\bar H]$, see~\cref{fig:two-cases}.

\begin{figure}[h!]
\begin{center}
\includegraphics[width=15cm]{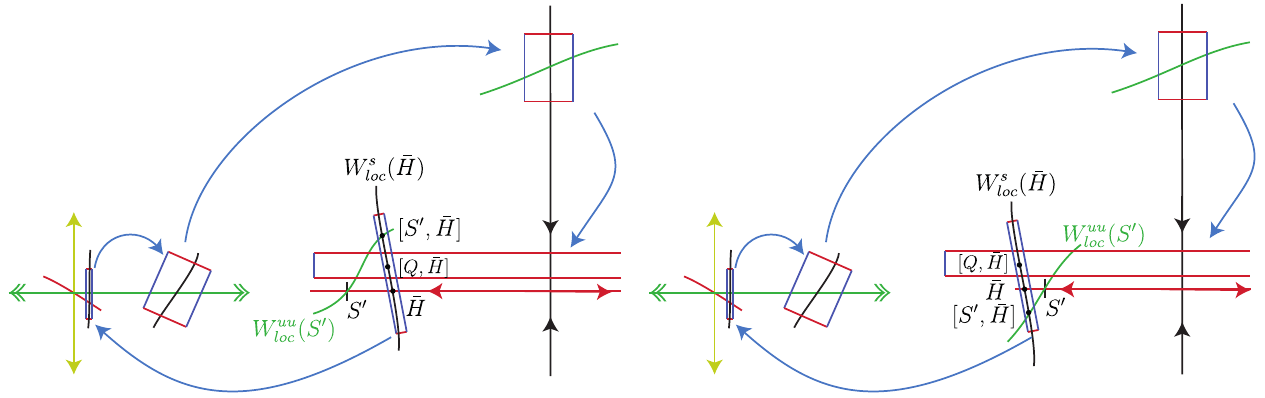}
\caption{The two cases for the position of $[S',\bar H]$.
\label{fig:two-cases}}
\end{center}
\end{figure}

We can reduce to the case depicted on the left part of \cref{fig:two-cases},
where $[S',\bar H]$ belongs to the half upper plane $\{y>0\}$ (for the chart of $V_P$).
Indeed if we are in the other case (depicted on the right part of \cref{fig:two-cases}), we use the fact that the stable branch
$\{0\}\times [-1,0]$ of $P$ has backward iterates which accumulate on $W^s_{loc}(P)$
in order to replace $H$ by a point $H'=(0,h')$, $h'<0$, which is a transverse intersection
between $W^s(P)$ and $W^{uu}(S)$. The new point $[S',\bar H']$ is close to $[S',\bar H]$,
hence belongs to the lower half plane.
It remains to conjugate the chart by $(x,y)\mapsto (x,-y)$ in order to find the desired configuration.
\medskip

Let us consider some large integers $n,N$,
the map $\mathcal{G}_2$ and the box $B$ defined at Section~\ref{Horseshoes  linked to the heterocycle}.
The transversality conditions $(T_2)$ and $(T_3)$ imply that
$\TH(W^{uu}_{loc}(S))$ crosses the box $\mathcal{G}_2(B)$ along
a small curve whose vertical coordinate belongs to an interval
$[c_1. \sigma_{uu}^n, c_2. \sigma_{uu}^n]$, where $c_1,c_2$
are independent from the choice of $n,N$.

We choose $n,N$ such that
\begin{equation}\label{choice-nN}
(h-\varepsilon)\lambda^{-N-1}< c_2 \sigma_{uu}^n<(h-\varepsilon)\lambda^{-N}.
\end{equation}
Note that the condition $\varepsilon\sigma^n_u\lambda^{N}\gg 1$ is satisfied
and Lemma~\ref{condition eigen B} associates a saddle point $\bar Q\in B$
whose vertical coordinates is in $[(h-\varepsilon)\lambda^{-N}, (h+ \varepsilon)\lambda^{-N}]$.
By the previous estimates, $\bar Q$ is ``above" the graph $\TH(W^{uu}_{loc}(S))$.

Now we consider a family $(f_t)_{t\in [0,1]}$ such that $f_0=f$, and for every $t$, the restrictions of $f_t$   to
$W^u_{loc}(P)$ and to the complement of a neighborhood of $V_P'$ coincide with $f$, while the restriction of the map $f_t|V_P'$ is still linear with eigenvalues  $(\lambda_t, \sigma)$ such that:
 \[\lambda_t= \frac{\lambda}{ \sqrt[N]{ 1+t \cdot (C-1)}}\quad \text{ with }  C= \frac{c_1}{c_2}\frac{h-\varepsilon}{h+\varepsilon} \cdot \lambda^{-1} \; .\]

Note that $(f_t)_{t\in [0,1]}$ is a smooth family which is $C^\infty$-close to be constantly equal to $f$ since $n$ is large. The map $\cS$, $\TS$, $\TH$ are unchanged, while $\cP_t^p:=  (x,y)\in V_P\mapsto  (\sigma \cdot  x, \lambda_t \cdot  y) $ depends on $a$. 
Any map of this family satisfies the assumptions of \cref{Horseshoes  linked to the heterocycle}.  Let $(\bar Q_t)_{t\in [0,1]}$ be the  hyperbolic continuation of $\bar Q$.
The vertical coordinate of $\bar Q_1$ is bounded by
\[(h+\varepsilon)\lambda_1^{-N}= C \cdot (h+\varepsilon)\cdot \lambda^{-N}= (h-\varepsilon)\cdot  \lambda ^{-N-1}
\frac{c_1}{c_2},\]
From~\eqref{choice-nN}, it is smaller than $c_1. \sigma_{uu}^n$, hence $\bar Q_1$ is ``below"
the graph  $\TH(W^{uu}_{loc}(S))$.
One deduces that there exists a parameter such that
$\bar Q_t$ belongs to $\TH(W^{uu}_{loc}(S))$.
This implies that $\bar Q_t$ has an iterate $Q$ which belongs to $W^{uu}_{loc}(S)$.
 \end{proof}
\begin{proof}[Proof of \cref{pingpong} in the $C^\infty$ case]
One considers a basic set $K$ provided by Corollary~\ref{c.basic}.
It contains a periodic saddle $P'$ homoclinically related to $P$ such that both of its stable branches intersects $W^{u}(P')$ transversally.
The Lemma~\ref{l.replace} allows by a first perturbation $\tilde f_1$ to replace $P$ by the saddle $P'$
so that the assumptions of the Lemma~\ref{l.pingpong} are satisfied.
One can then build a new perturbation $\tilde f_2$ such that $W^{uu}(S,\tilde f_2)$ contains a saddle
$Q$ which is homoclinically related to $P$ and $P'$, whereas the heterocycle
between $S$ and $P'$ is not destroyed (since the perturbation does not modify $S$ nor $W^u_{loc}(P')$).
After a third perturbation $\tilde f_3$ provided by Lemma~\ref{l.replace},
a strong heterocycle between $Q$ and $S$ is obtained.
\end{proof}

\subsection{Proof of \cref{pingpong} in the analytic case}

Now we assume $f\in \Diffloc^\omega(U,M)$ and as before $f$ displays a heterocycle between a saddle $P$
and a source $S$. To prove  \cref{pingpong} in the analytic case,  it suffices to show the following counterparts of
\cref{l.replace,l.pingpong}.

\begin{lemma}\label{l.replace-omega}
Let $Q$ be a periodic saddle point that is homoclinically related to $P$.
Then there exists a map $\tilde f$ that is $C^\omega$ close to $f$
such that $S$ and $Q$ form a heterocycle.

If $W^{uu}(S;f)$ contains $Q$, then, one can choose $\tilde f$ so that
$S$ and $Q$ form a strong heterocycle.
\end{lemma}

\begin{lemma}\label{l.pingpong-omega}
Let us assume that both stable branches of $P$ intersect $W^u(P)$ transversally.
Then there exists a map $\tilde f$, $C^\omega$-close to $f$,
with a saddle $Q$ homoclinically related to $P_{\tilde f}$
such that:
\begin{itemize}
\item $W^{uu}(S,\tilde f)$ contains $Q$.
\item $W^u(P,\tilde f)$ contains $S$.
\end{itemize}
\end{lemma}
\begin{proof}[Proof of \cref{l.replace-omega}]
First recall that $M$ is analytically embedded into an Euclidean space $\R^N$, see  \cite{Gr58}. Hence 
 there exists an analytic retraction $\pi :  U\to  M$ of a neighborhood $U$ of $M$ in $\R^N$.
Let $W^u_{loc}(P)$ be a local unstable manifold of $P$ which contains $S$ in its interior and
let $S'\neq S$ in $W^u_{loc}(P)$ such that $f(S')=S$.
Let $V_{S'}$ be a small neighborhood of $S'$ such that
the backward orbit of $S'$ inside $W^u_{loc}(P)$ does not meet $S'$.
One takes an analytic chart  $\phi: V_{S'}\to [-1,1]^2$ sending 
 $S'$ to $0$ and $V_{S'}\cap W^u_{loc}(P)$ to $[-1,1]\times \{0\}$.

Now consider a $C^\infty$-family $(f_{p})_{p \in [-\varepsilon,\varepsilon]} $ such that  $f_0=f$ and each $f_p$ is equal to $f$ outside $V_{S'}$ while on a  smaller neighborhood of $S'$, the map $f_p$ coincides with the composition of $f$ with a translation of vector $(0,p)$. In particular the continuation of $W^u_{loc}(P)$ for $f_p$ inside $V_{S'}$ is equal to $W^u_{loc}(P)$, while the continuations $S_p$ of $S$
and of its preimage $S'_p=f^{-1}(S)\cap V_{S'}$ satisfy that
$\partial_p S'_p|_{p=0}$ has non-zero second coordinate.  
Remark that $\chi := Df^{-1} \circ (\partial_p f_p|_{p=0}) $ is a  smooth vector field defined on the compact subset $\bar U\subset \R^N$. 
Then by Stone-Weierstrass Theorem,
there exists a  polynomial vector fields $\tilde \chi \in \R[X_1, \dots, X_N]$  whose restriction to $\bar U$ is
$C^1$-close $ \chi$. Also by reducing  $\varepsilon>0$,  the following is well defined for any $|p|<\varepsilon$: 

\[\tilde f_{p}:=  x\in U\mapsto \pi\big( f(x) +p\cdot  Df\circ  \tilde \chi  (x)\big)\; .\]
Note that  $\partial_p \tilde f_p|_{p=0}=Df\circ \tilde \chi$ is
$C^1$-close to  $\partial_p f_p|_{p=0}$.  In particular 
the hyperbolic continuation  $(\tilde S'_p)_{p\in [-\varepsilon, \varepsilon]}$ of $S'$ for $(\tilde f_p)_p$ is family
$C^1$- close to $(\tilde S'_p)_{p\in [-\varepsilon, \varepsilon]}$. Also the hyperbolic continuation $(W^u_{loc}(P, \tilde f_p))_{p\in [-\varepsilon, \varepsilon]}$ is a family of curves
$C^1$- close to the family constantly equal to $[-1,1]\times \{0\}$. 
Hence assuming that the $C^1$-size of the perturbation is small,  the curve $\Gamma:= \bigcup_{p\in [-\varepsilon, \varepsilon]} \{\tilde S'_p\}\times \{p\}$ intersects transversaly the surface 
$\Sigma:= \bigcup_{p\in [-\varepsilon, \varepsilon]} W^u_{loc}(P, \tilde f_p)\times \{p\}$ at   $\{S'\}\times \{0\}$. 

 By the inclination lemma with parameter, see \cite[Lemma 3.2]{berger2017emergence}, there exists a sequence $(W_{n,p})_n$ of  $p$-families of segments $W_{n,p}\subset W^u(Q,f_p)  $ such that  the sequence of surfaces 
$\Sigma_n := \bigcup_{p\in [-\varepsilon, \varepsilon]} W_{n,p}\times \{p\}$ converges to $\Sigma$ in the $C^1$-topology as $n\to \infty$. 
Thus when  $n$ is large,   the curve $\Gamma$ intersects $\Sigma_n$ at a point   close to $\{S'\}\times \{0\}$.   Hence there is $p$ arbitrarily small such that the continuations of $S'$ and $Q$ form a heterocycle for  $\tilde f_p$.  This proves the first part of the lemma since $\tilde f_p$ is $C^\omega$-close to $f$ when $p$ is small. 

In the second part of the lemma,  the saddle $Q$ belongs to a local strong unstable manifold $W^{uu}_{loc}(S)$  of $S$
and one performs a similar construction.
Let $Q'\neq Q$ in $W^{uu}_{loc}(S)$ which satisfies $f(Q')=Q$, let $V_{Q'}$ be a small neighborhood of $Q'$,
and consider a chart $\psi\colon V_{S'}\to [-1,1]^2$ sending 
 $Q'$ to $0$ and $V_{Q'}\cap W^{uu}_{loc}(Q)$ to $[-1,1]\times \{0\}$.
One considers a $C^\infty$ family of maps which are equal to $f$ outside $V_{Q'}$
and which coincide with the composition of $f$ with a translation of vector $(0,q)$ on a small neighborhood of $Q'$:
it induces a vector field $\xi$, that can be approximated by a polynomial vector field $\tilde \xi$.
Up to shrinking $\varepsilon>0$, for every $(p,q)\in [-\varepsilon, \varepsilon]^2$, the following is well defined:
\[\tilde f_{p,q}:=  x\in M\mapsto \pi\big( f(x) +p\cdot  Df\circ  \tilde \chi  (x)+q\cdot  Df\circ  \tilde \xi  (x)\big)\; .\]
Similarly we can consider the continuation $\tilde S_{p,q}$ of $S$, $\tilde Q_{p,q}$ of $Q$,
$W^u_{loc}(P, \tilde f_{p,q})$ of  $W^u_{loc}(P, f)$,  $W_{n,p, q}$ of $W_{n,p}$ and $W^u_{loc}(Q, \tilde f_{p,q})$ of $W^u_{loc}(Q)$.

From the first part of the proof,
$W^{u}_{loc}(P,\tilde f_{p,q})$ contains $\tilde S_{p,q}$ when $(p,q)$ belongs to graphs $\gamma_n$
that are arbitrarily $C^1$-close to the curve $p=0$ when $n\to \infty$.
By a similar argument, $W^{uu}_{loc}(S,\tilde f_{p,q})$ contains $\tilde Q_{p,q}$ when $(p,q)$
belongs to a one-dimensional submanifold $\sigma$ that contains $0$, is $C^1$-close to the curve $q=0$.
In particular $\sigma$ is transverse to the graphs $\gamma_n$.
Thus the conclusion of the lemma holds for some map $\tilde f_{p,q}$
with $(p,q)\in\gamma_n\cap \sigma$ which is $C^\omega$-close to $f$
when $n$ is large and $p,q$ are small.
This implies  the second part of the lemma. \end{proof}
 
 \begin{proof}[Proof of \cref{l.pingpong-omega}]
 The proof of \cref{l.pingpong} was obtained using a smooth family which changes the stable eigenvalue of $P$, without changing the relative position of $S$ w.r.t. $W^u_{loc}(P; f)$.   To obtain the analytic setting, as above, we approximate this family by an analytic one and we add an extra parameter which varies the relative position of $S$ w.r.t. $W^u_{loc}(P; f)$. While the first parameter enables to find a saddle $Q$ homoclinically related to $P$ such that $Q\in W^{uu}_{loc} (S)$, in the analytic setting this unfolding might unfold also the heterocycle. However the new second parameter enables to restore it.  \end{proof}

 \subsection{Proof of \cref{Ppingpong}: from paraheterocycles to strong paraheterocycles}
We follow the proof of the Proposition~\ref{pingpong} in the $C^\infty$ case.
After a first $C^\infty$-small perturbation of $f_0$ (and hence of the family $(f_a)_{a\in \R}$),
there exists a saddle $Q$ homoclinically related to $P$ which belongs to $W^{uu}_{loc}(S)$.
The paracycle property~\eqref{e.paracycle} between $S$ and $P$ may not hold anymore,
but by a new perturbation, with a similar size, it can be restored.
Note that it is supported near $f^{-1}(S)\setminus \{S\}$, hence the property $Q\in W^{uu}_{loc}(S)$ is not destroyed.
Finally one applies lemma~\ref{l.replace-para}, and gets a $C^\infty$-small perturbation
of the family $(f_a)_{a\in \R}$ in order to get a strong $C^r$-paraheterocycle at $a=0$
between $S$ and $Q$.
\qed

\section{From chains of heterocycles to paraheterocycles}\label{s.heterocycles}
We prove \cref{P.GenGaraheteroPara} in this section:
an $N$-chain of alternate heterocycles
whose saddles are homoclinically related, can be perturbed as a
$C^d$-paraheterocycle, provided that $N$ is large enough with respect to $d$. This is shown by induction on $d$.
The case $d=0$ follows from the continuity of the family (without any perturbation).
The induction step is given by: 

\begin{proposition}\label{theprop}
Consider a $C^\infty$ family
$(f_a)_{a\in \R}$  in $\Diffloc^\infty(U,M)$ and $d\geq 0$
such that $f_0$ has a $2$-chain of alternate heterocyles with saddle points $P^1,P^2$ and sources $S^1,S^2$
such that $(P^1,S^1)$ and $(P^2,S^2)$ form two $C^d$-paraheterocycles at $a=0$.
Then there is a $C^\infty$-perturbation of $(f_a)_{a\in \R}$ such that the continuation of $(P^1,S^2)$ forms a  $C^{d+1}$-paraheterocycle at $a=0$.

Moreover the perturbation is supported on a small neighborhood of
$ \text{\rm orbit}(S^1)\cup \text{\rm orbit}(S^2)$.
\end{proposition}
\begin{proof}[Proof of \cref{P.GenGaraheteroPara}]
One considers a $2^d$-chain of alternate heterocycles with
periodic points $P^1,S^1,\dots,P^{2^d},S^{2^d}$. \cref{theprop}
allows to perform a perturbation at $ \text{\rm orbit}(S^1)\cup \text{\rm orbit}(S^2) $,
such that $P^1$ and the continuation of $S^2$ form a $C^1$-paraheterocycle.

Note that $P^1,S^2,P^3,S^3,\dots,P^{2^d},S^{2^d}$ is still a $2^d-2$-chain of alternate heterocycles.
By induction, one gets a $2^{d-1}$-chain of alternate heterocycles
$P^1,S^2,\dots,P^{2^d-1},S^{2^d}$ such that $P^{2i+1},S^{2i+2}$ form a $C^1$-paraheterocycle at $a=0$,
for each $0\leq i< 2^{d-1}$.

By a new perturbation supported near the sources, one gets
a $2^{d-2}$-chain of alternate heterocycles
$P^1,S^4,\dots,P^{2^d-3},S^{2^d}$ such that each pair $P^{4i+1},S^{4i+4}$ forms a $C^2$-paraheterocycle at $a=0$. Repeating this construction inductively,
one gets a $C^d$-paraheterocycle at $a=0$ between $P^1$ and the continuation of $S^{2^d}$.
\end{proof}

\cref{theprop} is proved in the next two subsections.
In \cref{ss.k-para} we discuss the case where there are several parameters.

\subsection{Notations and local coordinates}
The setting is similar to \cref{Local coordinate for a heterocycle} and depicted \cref{notation}.
We chooses a large integer $r$ and a small number $\varepsilon>0$, we look for  a smooth perturbation of 
 $(f_a)_{a\in \R}$ which is $\varepsilon$-$C^r$-small and  such that the continuation of $(P^1,S^2)$ forms a  $C^{d+1}$-paraheterocycle at $a=0$.

As in \cref{s.strong} we shall assume that the points $P^2$ and $S^1$
are fixed. 
We denote by $|\sigma_a|<1$ and $  \lambda_a<-1$
(resp. by $|\sigma^{uu}_a| < |\sigma_a ^u| < 1$)
the inverse of the eigenvalues of the tangent map of $f_a$ at $P^2_a$ (resp. at $S^1_a$).

After a small perturbation we can assume that the eigenvalues are non-resonant and:
\[\frac {\log|\sigma^u_0|}{\log|\lambda_0|}\in \R\setminus \mathbb Q\; .\]

Then by  \cite{Ta71},    there exist:
\begin{itemize} 
\item  neighborhoods $V'_S(a)\subset V_S(a):= f_a(V'_S(a))$ of $S^1_a$ endowed with coordinates depending $C^r$ on the parameter and for which the inverse branche   $\cS_a:= (f_a|V_S')^{-1}$ has the form:
\[\cS_a:  (x,y)\in V_S\mapsto (\sigma^{uu}_a  \cdot  x, \sigma^u_a \cdot  y)\in V_S'\]
\item    neighborhoods $V'_P(a)$ and $V_P(a):= f_a(V_P'(a))$  of $P_a^2$ endowed with coordinates depending $C^r$ on the parameter and  for which the inverse branch   has the form:
\[ \cP_a:  (x,y)\in V_P(a)\mapsto  (\sigma_a \cdot  x, \lambda_a \cdot  y)\in  V_P'(a)\; .\]
\end{itemize}
Up to restricting  $V_P$, $V_P'$ and $V_S$ 	we can assume them equal to filled rectangles containing $0$ in their interior.
We define:
\[W^u_{loc}   (P^2_a)\equiv  V_P(a)\cap \{y=0\} \; ,\quad W^s_{loc}   (P^2_a)\equiv  V'_P(a)\cap \{x=0\}\qand W^{uu}_{loc}(S^1_a)\equiv  \{y=0\} \cap V_S(a)\; .\] 

Let $H_0$ be a point in $ W^s(P_0^2)\cap W^{uu}(S_0^1)$. Up to replacing it by an iterate, we can assume that $H_0$ belongs to $V'_P(0)$ with $H_0\equiv (0,h_0)$  in the linearizing coordinates of $P_0^2$. 

Also, a preimage $S'^2_a$ of $S^2_a$ by an iterate of  $f_a$   has coordinates $S'^2_a =:(x_a,y_a)$ in  the linearizing coordinates of $P^2_a$.  Let  $\cT_a: V_H\hookrightarrow V_S$ be an inverse branches of an iterate of $f_a$ defined on a neighborhood $V_H \Subset V'_P$  of $H_0$ and such that  $\cT_0$ sends  $H_0$ into $W^{uu}_{loc}(S^1_0)$. 

Up to a smooth perturbation, one can require that:
\begin{enumerate}[$(T_1)$]
\item $W^u(P^1)$ is transverse to $E^{cu}_{S^1}$ at $S^1$,
\item $W^{uu}_{loc} (S^1_0)$ and $W^s_{loc}(P^2_0)$ are transverse at $H_0$.
\end{enumerate}
By $(T_1)$,   $W^u(P^1_a)$ contains a graph in the chart at $S^1_a$, over a neighborhood
$I\subset \R$ of $0$:
\[\Gamma_a \equiv  \{(x, \gamma_a(x));\; x\in I\}  .\]
By $(T_2)$, the transverse intersection $H_0$ admits a continuation $H_{a}$ for $a$ close to $0$.
One sets
$$H_a\equiv (0,h_a) \qand \cT_a(H_a)=(z_a, 0)
 \, .$$
Since $(P^1,S^1)$ and $(P^2,S^2)$ form two $C^d$-paraheterocycles at $a=0$,
one has for any $0\leq k\leq d$,
\[\partial^k_a \gamma_a(0)|_{a=0}=0
\quad \text{and} \quad \partial^k_a y_a|_{a=0}=0.\]
Up to a small perturbation,
one can also assume that
$$\partial^{d+1}_a \gamma_a(0)|_{a=0}\neq 0
\quad \text{and} \quad \partial^{d+1}_a y_a|_{a=0}\neq 0.$$
Figure \ref{notation} summaries the notations.
\begin{figure}[h]
    \centering
        \includegraphics[width=12cm]{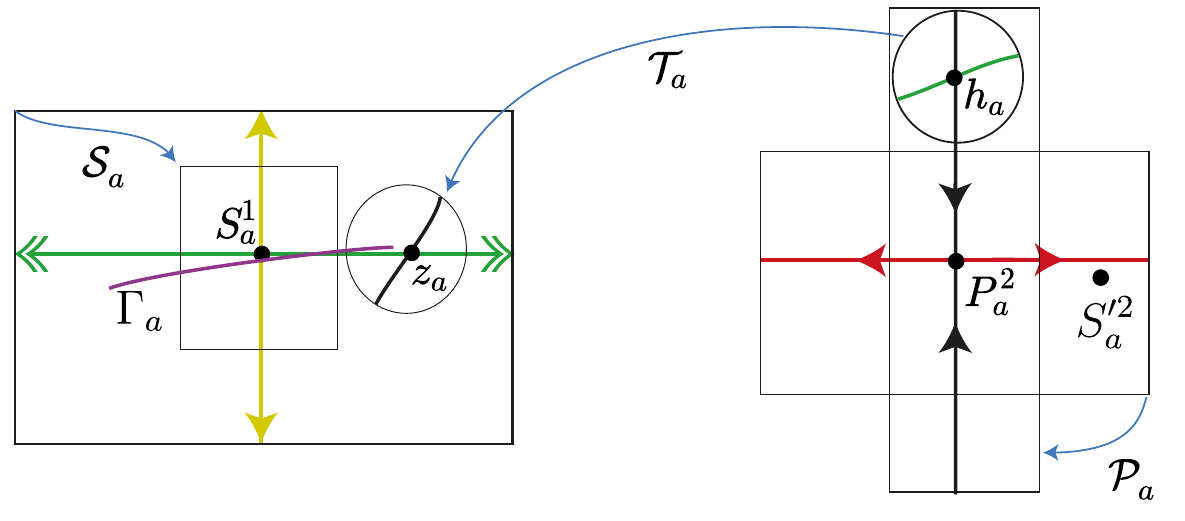}
    \caption{Notations.}
    \label{notation}
\end{figure}

\subsection{Compositions nearby a paraheterocycle}
Let  $\Delta$ be the second coordinate of $ \partial_y  \cT_0(H_0)$; it is nonzero by $(T_2)$.

\begin{lemma}\label{l.composition}
Given integers $n,m\geq 1$ large such that  $(\sigma^{u}_0)^n\lambda_0^m=O(1)$, there is a $C^r$-perturbation of $(f_a)_a$ locallized at $S^2_a$  such that the germ at $a=0$ of  $a\mapsto \cS_a^n\circ \cT_a\circ \cP_a^m(S'^2_a)$   is $C^{d+1}$-close to 
\[ 
\left( 0,(\sigma^{u}_0)^n\cdot \Delta \cdot   \lambda^m_0 \cdot \frac{ \partial_a^{d+1} y_a|_{a=0} }{(d+1)!} a^{d+1} \right).\]
\end{lemma}

\begin{proof}
For $m$ large, after a $C^r$-small perturbation localized at $S_a^2$ (which is conjugated to a translation in a small neighborhood of $S_a^2$), we can assume $S'^2_a=(x_a, y_a+\varepsilon_a)$ where $a\mapsto \varepsilon_a$ is the $C^\infty$-small function defined by $\varepsilon_a :=  \lambda_a^{-m}\cdot h_a$ and where as before $(x_a,y_a)$ are the coordinates of $S'^2_a$ before the perturbation.

Then observe that  
$\cP_a^m(S'^2_a)= H_a+(\sigma_a^{m} \cdot x_a, \lambda_a^m \cdot  y_a)$ forms a   family whose germ at $a=0$ is   $C^{d}$-close to $(H_a)_a$. 
When $m$ is large, the germ at $a=0$ of  $a\mapsto \cT_a\circ \cP_a^m(S'^2_a)$   is $C^{d+1}$-close to 
\[\cT_a(H_a)+ D_{H_a} \cT_a \left( \sigma_a^{m} \cdot x_a, \lambda^m_a \cdot \frac{ \partial_a^{d+1} y_a|_{a=0} }{(d+1)!} a^{d+1} \right)
 \]
and so $C^{d+1}$-close to 
\[(z_a, 0)+ \partial_y  \cT_0(H_0) \cdot   \lambda^m_a \cdot \frac{ \partial_a^{d+1} y_a|_{a=0} }{(d+1)!} a^{d+1} 
\; .\]
Consequently, for any $n\ge 0$,  the germ at $a=0$ of  $a\mapsto \cS_a^n\circ \cT_a\circ \cP_a^m(S'^2_a)$   is $C^{d+1}$-close to 
\[((\sigma_a^{uu})^n\cdot z_a, 0)+ \mathrm{diag }((\sigma^{uu}_a)^n,(\sigma^{u}_a)^n )\cdot \partial_y  \cT_0(H_0) \cdot   \lambda^m_a \cdot \frac{ \partial_a^{d+1} y_a|_{a=0} }{(d+1)!} a^{d+1} 
\; .\]
If   $(\sigma^{u}_0)^n\lambda_0^m=O(1)$, then both   $(\sigma_a^{uu})^n$ and  $(\sigma^{uu}_0)^n\lambda_0^m$ are small, and so we obtain the announced bound. 
\end{proof}

Since the ratio $\log|\sigma^u_0|/\log|\lambda_0|$ is irrational,
and since $\partial^{d+1}_a \gamma_a(0)|_{a=0}\neq 0$
and $\partial^{d+1}_a y_a|_{a=0}\neq 0$,
one can choose some large positive integers $n,m$ such that
$$n\log|\sigma^u_0|+m\log|\lambda_0|-\log|\Delta|+
\log|\partial^{d+1}_ay_a|_{a=0}
$$
is arbitrarily close to $\log|\partial^{d+1}_a\gamma_a(0)|_{a=0}$. Since $\lambda$ is negative, one can choose $m$ to be odd or even
so that
$\Delta\cdot  (\sigma_0^u)^n(\lambda_0)^m\partial^{d+1}y_a|_{a=0}$
and $\partial ^{d+1}\gamma_a(0)|_{a=0}$ have the same sign.

By our assumptions, the $C^d$-jets of $a\mapsto \gamma_a(0)$ and $a\mapsto y_a$ at $a=0$ vanish.
With our choices, this guaranties that the $C^{d+1}$-jet at $a=0$ of
$a\mapsto \Delta \cdot (\sigma_0^u)^n(\lambda_0)^my_a-\gamma_a(0)$ is arbitrarily small.
By~\cref{l.composition}, after a $C^r$-perturbation of $(f_a)_a$ localized at $(S^2_a)_a$, the germ at $a=0$ of the following function is $C^{d+1}$-small:
\[a\mapsto \eta_a:=  \gamma_a\circ  p_x\circ  \cS_a^n\circ \cT_a\circ \cP_a^m(S'^2_a) -p_y\circ  \cS_a^n\circ \cT_a\circ \cP_a^m(S'^2_a).\]
A $C^r$-small perturbation localized at $S_a^1$ (which is locally conjugated to a translation)
translates the functions $(\gamma_a)_a$ by $-\eta_a$ for each parameter $a$ close to $0$.
Then we have at $a=0$: 
$$d(\Gamma_a, \cS_a^n\circ \cT_a\circ \cP_a^m(S'^2_a))= o(  a^{d+1}).$$

As a consequence, the continuation of $P^1_0$ and $ S^2_0$ form a $C^{d+1}$-paraheterocycle at $a=0$ for the chosen perturbation.  
Since the charts are a priori only $C^r$, the resulting perturbation is only $C^r$.
In a last step, we thus smooth the family  near the sources,  keeping the paraheterocycle we have obtained
(the latter being a finite codimensional condition on the family).  \cref{theprop} is now proved. \qed

\subsection{Families parametrized by $k$-parameters}\label{ss.k-para}
When the family $(f_a)$ is parametrized by $a=(a_1,\dots,a_k)$ in $\R^k$, $k>1$, the proof follows the same scheme, by canceling one by one the partial derivatives 
$\partial_{a_1}^{i_1}\partial_{a_2}^{i_2}\cdots \partial_{a_k}^{i_k}$ of the unfolding of the heterocycle.
For this end, we proceed by induction on $\{\underline i= (i_1,\dots, i_k)\in \N ^k:\; \sum_j i_j \le d\}$  following an order $\prec$ such that:
\[\sum_j i_j < \sum_j i_j'\Rightarrow \underline i \prec \underline i'.\]

\section{Nearly affine (para)-blender renormalization}\label{s.blender}
In this section, we prove \cref{PPaffine,PPPaffine}.

We consider a $C^\infty$ map $f\in \Diffloc^\infty(U,M)$ with a projectively hyperbolic source $S$ and a saddle point $Q$ forming a strong heterocycle, and build by perturbation a nearly affine blender homoclinically related to $Q$. It is defined by two inverse branches  from a neighborhood of $Q$ to  ``vertical rectangles'' stretching across the local unstable manifold of the saddle. 

In \textsection\ref{Setting modulo small perturbation} and \textsection\ref{Unfolding of the strong heterocycle}
we choose nice coordinate systems for the inverse dynamics nearby the source, the saddle and the heteroclinic orbits. It requires preliminary perturbation in order to satisfy some non-resonance and transversality conditions. We also explain how to unfold the strong heterocycle.
The heterocycle induces well-defined inverse branches of the dynamics (\textsection \ref{Choosing candidates}) that are transitions from one linearizing chart to the other.
\textsection \ref{bound} provides $C^r$-estimates on rescalings $g^-,g^+$ of the inverse branches.
In \textsection \ref{Tuning} and \textsection \ref{Translating candidates}, we tune the length of the branches and the size of the unfolding so that the inverse branches
defines a nearly affine blender with a neat dilation $\Delta$; it is homoclinally related to the saddle point $P$ and that its activation domain contains $S$.  In other words, \cref{PPaffine} will be proved.

In \textsection\ref{sectionparadense}, we add a parameter, consider a family $(f_a)_{a\in\R}$
and apply the previous discussion to $f_0$. The inverse branches admit continuations $(g^-_a)_{a\in\R}$ and $(g^+_a)_{a\in\R}$.
After having chosen an adapted reparametrization, we extend the $C^r$-bounds to the parametrized families and check that
they define a nearly affine $C^r$-parablender, concluding the proof of~\cref{PPPaffine}.

\paragraph{\it Notations.}
The proofs will depend on a small number $\varepsilon>0$ and on integers $n^+,n^-,m^+,m^-$.
The notation $A=O(\varepsilon)$ (or more generally $A=O(g(\varepsilon, n^+,n^-,m^+,m^-))$)
will mean that the quantity $A$ has a norm bounded by $C.\varepsilon$
(or by $C.|g(\varepsilon, n^+,n^-,m^+,m^-)|$), where the number $C>0$ depends on the initial map $f$
but not on the choices made during the construction.

Similarly, one will say that a function $h$ (that may depend on coordinates $x,y$, and/or parameters $a$ or $\alpha$)
is \emph{$C^r$-dominated} by $\varepsilon$ if $\partial^k h= O(\varepsilon)$
for all its $k^\text{th}$ derivatives with respect to $x,y,a,\alpha$ with $0\leq k\leq r$.
Note that if in the $C^r$-topology, $h_i=h'_i+O(\varepsilon)$, $i\in \{1,2\}$, then $h_1\circ h_2=h'_1\circ h'_2+O(\varepsilon)$.

\subsection{Coordinates   for generic perturbations of strong heterocycles}\label{Setting modulo small perturbation}
We first fix a system of coordinate as depicted in \cref{fig:def_inversebranches2}.
As in \cref{Local coordinate for a heterocycle}, we shall assume that the points $Q$ and $S$
are fixed and the eigenvalues $1<\sigma_u^{-1}<\sigma_{uu}^{-1}$ and $0<\lambda^{-1}<1<\sigma^{-1}$
of $D_Qf$ and $D_Sf$ respectively are positive and non-resonant.
Furthermore we can assume that:
\begin{equation}\label{palis module irrational}   \frac{\log \sigma_u}{\log \lambda} \notin \mathbb Q\; .\end{equation}  
The hypothesis of the proposition consists of   two finite codimensional conditions: 
\begin{equation}\label{heterocycle-condition}
S\in W^u(Q; f) \qand Q\in W^{uu}(S; f)\; .
\end{equation}
So after a  small  smooth perturbation, we can assume moreover:
\begin{equation}\label{transversality strong hetro}  T_QW^{uu}(S; f)\oplus T_Q W^s(Q; f)=T_Q M\qand E^{cu}(S) \oplus T_S W^u(Q; f)= T_S M \; .\end{equation}

As in \cref{Local coordinate for a heterocycle}, the non-resonance of the eigenvalues and the smoothness of the dynamics imply, by the Sternberg Theorem \cite{S58}, the existence of:
\begin{itemize} 
\item  Neighborhoods $V'_S\subset V_S:= f(V'_S)$ of $S$ and coordinates for which $f|V'_S$ has the form:
\[f\colon (x,y)\in V'_S\mapsto (\sigma_{uu}^{-1}\cdot  x, \sigma_u^{-1}\cdot  y)\in V_S.\]
\item  Neighborhoods $V'_Q$ and $V_Q= f(V_Q')$  of $Q$ and  coordinates  in which  $f|V'_Q$ has the form:
\[f\colon (x,y)\in V'_Q\mapsto (\sigma^{-1}\cdot  x, \lambda^{-1}\cdot  y)\in V_Q.\]
\end{itemize}

This defines the inverse branches $\cQ:= (f|V_Q')^{-1}$ and 
 $\cS:= (f|V_S')^{-1}$:
\[\cS:  (x,y)\in V_S\mapsto (\sigma_{uu}  \cdot  x, \sigma_u \cdot  y)\in V_S'\qand \cQ:  (x,y)\in V_Q\mapsto  (\sigma \cdot  x, \lambda \cdot  y)\in  V_Q'\; .\]
Up to restrict  $V_Q$ and $V_Q'$ and rescale their coordinate, we can assume:
\[V_Q\equiv [-2,2]\times [-2\lambda^{-1},2\lambda^{-1}]\qand 
V_Q'\equiv [-2\sigma,2\sigma]\times [-2,2]
\; .\]
Let $W^u_{loc}   (Q):= V_Q\cap \{y=0\} $,   $W^s_{loc}   (Q):= V'_Q\cap \{x=0\} $ and $W^{uu}_{loc}(S):= \{y=0\} \cap V_S$.
\medskip

By \cref{heterocycle-condition} there is a neighborhood $V''_S \Subset V'_S$ of $0\equiv S$ and an inverse branch $\TS: V_S''\hookrightarrow V_Q$  of an iterate of $f$ sending $0$ into $[-2,2]\times \{0\}$. 
Similarly,  there exists a neighborhood $V''_Q \Subset V'_Q\cap V_Q$ of $0\equiv Q$ and 
an inverse branch $\TQ: V_Q''\hookrightarrow V_S$  of an iterate of $f$ sending $0$ into $V_S\cap \{y=0\}$. 
The inverse branches  $\TS$ and $ \TQ$ are called the \emph{transitions maps}.  
\begin{figure}[h!]
\begin{center}
\includegraphics[width=11cm]{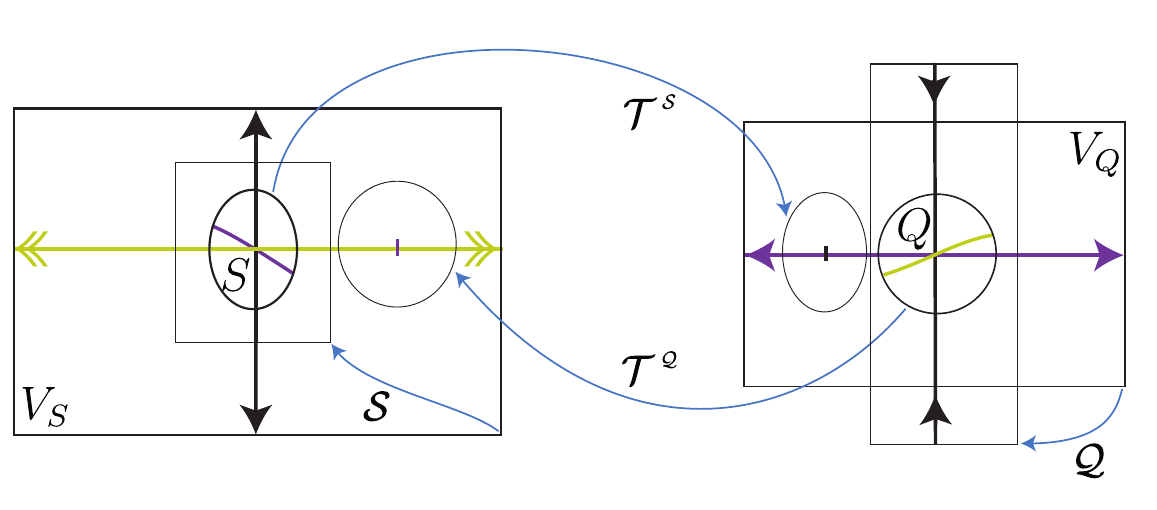}
\caption{Inverse branches given by the strong heterocycle.  \label{fig:def_inversebranches2}}
\end{center}
\end{figure}

\noindent
Assuming the neighborhoods $V_S$ and $V_Q$ small enough, it is possible (up to compose by an iterate of $f$) to choose
$\TS,\TQ$ such that $$\TS(0)\in V'_Q\setminus V_Q
\qand \TQ(0)\in V_S\setminus V'_S.$$
Let the coordinates of $\TS$ and $\TQ$ be 
\[\TS:=(\XS, \YS)\qand \TQ:=(\XQ, \YQ)\; .\]
By~\cref{transversality strong hetro},
$  \partial_y \YQ(0)\not= 0$. Thus by rescaling one of the linearizing chart, we can assume:
\begin{equation}\label{condition transition}
 \partial_y \YQ(0)=1.
\end{equation}

\subsection{Unfolding of the strong heterocycle}\label{Unfolding of the strong heterocycle}
We will perturb $\TS,\TQ$
so that the following points are close to but \emph{not necessarily} in $\{y=0\}$:
\[S'=(s'_x, s'_y):= \TS(0)\qand  Q'= (q'_x, q'_y):= \TQ(0)\; .\]
This is enabled by the next claim without changing any derivative of the inverse branches.  
\begin{claim} \label{unfolding as we want}
For every small numbers $s'_y$ and $q'_y$, there exists a $C^\infty$ perturbation of the dynamics such that the inverse branches $\cS$ and $\cQ$ remain unchanged, while the  continuations  of the inverse branches   $\TS$ and $\TQ$ have the same derivatives but satisfy:
\[\YS(0)=s'_y \qand \YQ(0)=q'_y.\]
\end{claim}
\begin{proof}
First recall that $\TS(0)\in V_Q\setminus V_Q'$.
One perturbs $f$ by composing with a translation supported on a small neighborhood of $\TS(0)$.
This enables to move the vertical position of  $ \TS(0)$, without affecting the other branches. The modification of  $\TQ(0)$ is done similarly. 
\end{proof}

In the following we will prescribe some values of $s'_y,q'_y$ and consider the perturbed dynamics.
The inverse branches of the new system will be still denoted by $\cQ$, $\cS$,  $\TS=(\XS,\YS)$ and $\TQ=(\XQ,\YQ)$. The next lemma enables to assume that  $\partial_y   \YS(0)$ is positive.
\begin{lemma}\label{partialy YSpositive} Up to perturbation $f$ 
and to change $\TS$, we can assume moreover that
$$\partial_y   \YS(0)>0.$$
 \end{lemma}
\begin{proof}
If $\partial_y   \YS(0)<0$, we are going to perturb $f$ and   replace $\TS$ by the inverse branch $\widetilde \TS:= \TS\circ \cS^{n } \circ \TQ\circ \cQ^{m } \circ \TS$ for some large $n$ and  $m$. First note that for any  large $n$ and any $m$,  
 the map $\widetilde \TS $ is well defined on a small neighborhood of $0\equiv S$. Also $\partial_y \TS(0)$ is a vector with negative vertical component. By hyperbolicity, it is sent by $D\cQ^m$ to a  vertical vector. Its vertical component is still negative since $\lambda>0$. It is pointed at a point $\cQ^{m } \circ \TS(0)$ close to $0\equiv S$. Thus for $m$ sufficiently large, by \cref{condition transition}, its image by $D\TS$ is a vector with negative vertical component. By projective hyperbolicity of the source $S$, its image by $D\cS$ is a vector  vertical, pointed at a point nearby $S$ when $n$ is large, and with negative vertical component. Consequently it is sent by $\TS$ to a vector with positive vertical component at a point nearby $\TS(0)$. In other words, the second coordinate of $\partial_y \widetilde \TS(0)$ is positive. 

It remains to perform a perturbation of $f$   so that the second coordinate of $\widetilde \TS(0)$ is zero. Let  $(\TS_t)_t$ be the family of  perturbations of $\TS$ given by \cref{unfolding as we want} and enabling to move the $y$-coordinate of $\TS(0)$. Note that when $n\gg m$,
\[\partial_t (\TS_t\circ \cS^{n } \circ \TQ\circ \cQ^{m } \circ \TS_t)(0)\approx \partial_t \TS_t(0) + D(\TS\circ  \cS^{n } \circ \TQ\circ \cQ^{m })(\partial_t \TS_t(0))\approx \partial_t \TS_t(0)\; .\]
Hence there is a small  parameter $t$  such that  $ \widetilde {\TS_t}:=(\TS_t\circ \cS^{n } \circ \TQ\circ \cQ^{m } \circ \TS_t)$ satisfies moreover that the $y$-coordinate of $ \widetilde {\TS_t}(0)$ is $0$. 
\end{proof}

\subsection{Choice and renormalization of inverse branches}\label{Choosing candidates}
Let us fix $\Delta>1$ sufficiently close to $1$ so that a nearly affine blender of contraction $\Delta^{-1}$ is a blender by \cref{ are blender}.
The construction also depends on a small number $\varepsilon>0$ (it will measure the distance of the rescaled blender to an affine one)
and on large integers $n^-,m^-,n^+,m^+$ that will be chosen later.

The nearly affine blender will be displayed  in the neighborhood $V_Q$  of $Q$, using two inverse branches $g^+$ and $g^-$ of  different iterates   of $f$. We take them of the  form: 
\[g^\pm := \TS\circ \cS^{n^\pm} \circ \TQ\circ\cQ^{m^{\pm}} .\]
 \begin{figure}[h!]
\begin{center}
\includegraphics[width=11cm]{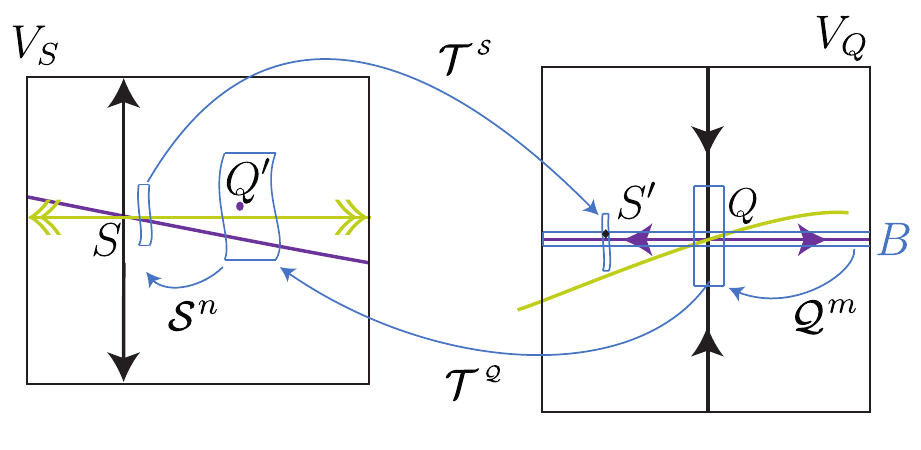}
\caption{Construction of a nearly affine blender.  \label{fig32}}
\end{center}
\end{figure}

The inverse branches defining the blender will be rescaled by the map:
\[\cH: (x,y)\in \R^2 \to
(x,\varepsilon \cdot \lambda^{-m^+} \cdot y)\in V_Q\; .\]
Their renormalization are given for $\pm\in \{-,+\}$ by:
  \begin{equation}\label{def Rg} \cR g^\pm:= \cH^{-1}\circ g^\pm  \circ \cH=  \cH^{-1}\circ \TS \circ \cS^{n^\pm}\circ  \TQ \circ \cQ^{m^\pm} \circ \cH\end{equation}
    \begin{lemma}\label{Rg well def}
For every $n^-,m^-,n^+,m^+$ large, with $m^+>m^-$, the renormalizations $\cR g^-,\cR g^+$ are well defined on $B:=  [-2,2] ^2$. 
\end{lemma}  
\begin{proof} Since $m^-< m^+$,
both maps $\cQ^{m^+}\circ \cH ,\cQ^{m^-}\circ \cH$ are well defined on $B$ and equal to:
  \[\cQ^{m^+} \circ \cH(x,y)= (  \sigma^{m^+}  \cdot  x,\varepsilon  \cdot y)\qand \cQ^{m^-} \circ \cH(x,y)= (  \sigma^{m^-}  \cdot  x,\varepsilon  \cdot \lambda^{m^--m^+} \cdot y)\; . \]
As $\varepsilon$ is small, their ranges are contained in a small neighborhood of $0$ and so in the domain of $\TQ$. Thus both maps  $ \TQ \circ \cQ^{m^\pm} \circ \cH$ are well defined on $B$ and their ranges lie in a small neighborhood of $\TQ(0)\in V_S$. 
 Then as $\cS$ contracts $V_S$ into itself with a fixed point at $0$ and since $n^\pm$ are large,
 $\cS^{n^\pm}\circ  \TQ \circ \cQ^{m^\pm} \circ \cH$ is well defined on $B$ and its image is included
 in the small neighborhood $V_S''$ of $0$. Thus $\TS\circ \cS^{n^\pm}\circ  \TQ \circ \cQ^{m^\pm} \circ \cH$ is well defined on $B$.
 \end{proof} 

\subsection{Bounds on the renormalized maps}\label{bound}

Given $\varepsilon>0$ small, we require the following properties on the large integers $n^\pm,m^\pm$:
\begin{gather}\label{e.hypo1}
n^+> n^- \ge \varepsilon^{-1} \qand  m^+>  m^-\ge   \varepsilon^{-1}\;,\\
\{\;\sigma_u^{n^-}\lambda^{m^-}\cdot \partial_y   \YS(0),\;
\sigma_u^{n^+}\lambda^{m^+}\cdot \partial_y   \YS(0)\;\} \; \subset \; 
[\Delta-\varepsilon,\Delta+\varepsilon]\;,\label{e.hypo3}\\
\varepsilon^{-1} \leq \sigma_u^{n^+-n^-} \qand \lambda^{m^+-m^-}\leq \varepsilon^2\cdot \min(n^-,m^-)\; . \label{e.hypo2}
\end{gather}
Let us recall that the inverse eigenvalues satisfy
$\kappa:= \max(\sigma_u ,\sigma_{uu} ,\frac{\sigma_{uu}}{\sigma_{u}}, \lambda^{-1},\sigma )<1.$
\begin{fact}\label{fact-eps}
   For every $\varepsilon>0$ small and $n>\varepsilon^{-1}$, it holds $\kappa^n<n^{-(r+4)}$.
\end{fact}
\noindent
In particular, one has $\sigma^{n^-}<\varepsilon$ and $\sigma_u^{n^-} <\varepsilon$.
\medskip

We decompose the renormalized maps as
\[\cR g^\pm =  \Psi^\pm \circ \Phi^\pm =
[\cH^{-1}\circ  \TS\circ \cS^{n^\pm} \circ  \cH^\pm]
\circ 
[(\cH^\pm)^{-1}\circ  \TQ \circ \cQ^{m^\pm} \circ \cH] \; ,\]
\[\text{where} \quad
\cH^\pm:= (x,y)\mapsto (x,\varepsilon \cdot \lambda^{m^\pm-m^+}\cdot y)-Q' .\]

\begin{lemma}\label{preprecondition fro blender}
The maps $(x,y)\mapsto \Phi^\pm(x,y)-(0,y)$
are $C^{r}$-dominated by $\varepsilon$.
\end{lemma}
 \begin{proof}
 We have $\Phi^\pm(x,y)=(\cH^\pm)^{-1}\circ  \TQ \circ \cQ^{m^\pm} \circ \cH(x,y)$.
 Since $Q'=\TQ(0)= \TQ\circ \cQ^{m^\pm}\circ \cH (0) $ we get $\Phi^\pm(0)=0$.
Recalling that $\TQ=(\XQ, \YQ)$ and that $Q'= (q'_x, q'_y)$, we obtain:
 \begin{eqnarray*}
 \Phi^\pm(x,y) &= &(  \XQ  ,  \varepsilon^{-1}  \cdot  \lambda^{m^+-m^\pm} \cdot  \YQ )(\sigma^{m^\pm}  \cdot   x
,\varepsilon  \cdot \lambda^{m^\pm-m^+} \cdot y)+ (  q'_x  ,  \varepsilon^{-1}  \cdot  \lambda^{m^+-m^\pm} \cdot  q'_y )\; .\\
\partial_x \Phi^\pm(x,y) &=&  \sigma^{m^\pm}  \cdot    (
   \partial_x \XQ  ,  \varepsilon^{-1}\cdot  \lambda^{m^+-m^\pm}\cdot  \partial_x \YQ )(\sigma^{m^\pm}  \cdot  x
   ,\varepsilon  \cdot \lambda^{m^\pm-m^+} \cdot y)  \; .\\ 
   \partial_y \Phi^\pm(x,y) &= &  (  \varepsilon\cdot   \lambda^{m^\pm-m^+}  \cdot   \partial_y \XQ  ,   \partial_y \YQ )(\sigma^{m^\pm}  \cdot  x
 ,\varepsilon  \cdot \lambda^{m^\pm-m^+} \cdot y)  \; .
 \end{eqnarray*}

From this, \eqref{e.hypo1} and \cref{fact-eps}, the first coordinate of  $D\Phi^\pm$ is $C^{r-1}$-dominated by $\varepsilon$.
Using also \eqref{e.hypo2}, we have
$\sigma^{m^\pm} \cdot \varepsilon^{-1}\cdot  \lambda^{m^+-m^\pm}<\sigma^{m^\pm} \cdot \varepsilon \cdot m^-<\varepsilon$
and the second coordinate of $\partial_x \Phi^\pm$ is $C^{r-1}$-dominated by $\varepsilon$.
As $ \partial_y \YQ(0)=1$ by \eqref{condition transition}, the second coordinate of $\partial_y \Phi^\pm$
coincides with the constant function $1$, up to an error term that is $C^{r-1}$-dominated by $\varepsilon$.
\end{proof}

\begin{lemma}\label{precondition fro blender}
The maps $\Psi^\pm$ coincide with
$$(x,y)\mapsto (0,\Delta\cdot y)+(s'_x\; ,\;  \varepsilon^{-1} \lambda^{m^+} \cdot s'_y- \varepsilon^{-1}\cdot \lambda ^{m^+}\sigma_u^{n^\pm } \cdot \partial_y   \YS(0) \cdot q'_y  ),$$
up to an error term that is $C^{r}$-dominated by $\varepsilon$.
\end{lemma}
\begin{proof}
We have $\Psi^\pm=  \cH^{-1}\circ  \TS\circ \cS^{n^\pm} \circ  \cH^\pm$.
With $\TS=(\XS, \YS)$, it holds:
 \[\Psi^\pm(x,y) = ( \XS, \varepsilon^{-1} \lambda ^{m^+} \YS )(\sigma_{uu} ^{n^\pm}\cdot (x-q'_x),\;  \sigma_{u} ^{n^\pm}
\cdot (\varepsilon \cdot \lambda^{m^\pm-m^+}\cdot y-q'_y))
\; . \]
Thus $\partial_x \Psi^\pm$ is $C^{r-1}$-dominated by $\sigma_{uu} ^{n^-}\cdot \lambda ^{m^+} \cdot \varepsilon^{-1}$,
which by~\eqref{e.hypo1}, \eqref{e.hypo3}, \eqref{e.hypo2} is dominated by
$$(\tfrac{\sigma_{uu}}{\sigma_u}) ^{n^-}\cdot \lambda^{m^+-m^-} \cdot \varepsilon^{-1}<(\tfrac{\sigma_{uu}}{\sigma_u}) ^{n^-} \cdot n^-\cdot \varepsilon<\varepsilon.$$
The first coordinate of $\partial_y \Psi^\pm$ is $C^{r-1}$-dominated by $\varepsilon  \cdot \sigma_u^{n^\pm} \cdot \lambda^{m^\pm-m^+}<\varepsilon$.
Similarly, using~\eqref{e.hypo3}, the second  coordinate of  $\partial_y \Psi^\pm$ coincides with $\sigma_u^{n^\pm} \cdot \lambda^{m^\pm } \cdot \partial_y \YS(0)$,
hence with $\Delta$, up to an error term that is $C^{r-1}$-dominated by $\varepsilon$.
We have thus shown that the derivative of $(x,y)\mapsto \Psi^\pm(x,y)-(0,\Delta\cdot y)$ is $C^{r-1}$-dominated by $\varepsilon$.
Moreover:
\[\Psi^\pm(0) =  (
\XS, \varepsilon^{-1} \cdot \lambda ^{m^+} \cdot  \YS )( - \sigma_{uu} ^{n^\pm}\cdot q'_x , -  \sigma_{u} ^{n^\pm}
\cdot  q'_y ) \; . \]
The first coordinate is $\varepsilon$-close to $\XS(0)=s'_x$ and the second coordinate is equal to:
\[ \varepsilon^{-1} \lambda ^{m^+} \YS (- \sigma_{uu} ^{n^\pm}\cdot q'_x, - \sigma_{u} ^{n^\pm}
\cdot q'_y))
 =\varepsilon^{-1} \lambda^{m^+}\bigg( \YS (0)
 - \partial_x \YS (0) \cdot \sigma_{uu} ^{n^\pm}\cdot q'_x
 - \partial_y \YS (0)\cdot \sigma_{u} ^{n^\pm}\cdot q'_y +O(\sigma_{u} ^{2n^\pm})\bigg).\]
 As before $\varepsilon^{-1}\cdot \lambda ^{m^+} \cdot \sigma_{uu} ^{n^-}= O(\varepsilon)$.
 By \eqref{e.hypo3} and \eqref{e.hypo2},
 $\varepsilon^{-1}\cdot\lambda^{m^+}\cdot \sigma_{u} ^{2n^-}$
 is dominated by
 $\varepsilon^{-1}\cdot\lambda^{m^+-m^-}\cdot \sigma_{u} ^{n^-}= O(\varepsilon).$
As $ \YS (0)=s'_y$, we obtain $(1)$. 
\end{proof}

\subsection{Tunning iterates}\label{Tuning}

\begin{lemma}\label{choixirrationel} Given $\varepsilon>0$ small, there exist  $n^-, m^-, n^+, m^+$ which satisfy \eqref{e.hypo1},
\eqref{e.hypo3}, \eqref{e.hypo2}.
\end{lemma}
\begin{proof} 
By~\eqref{palis module irrational}, there exist $m,n\ge 1$ arbitrarily large such that:
\[ \lambda^{m}\sigma_{u}^{n}\in [1-\tfrac \varepsilon{10}, 1+\tfrac \varepsilon{10}]\; .\] 
As $\Delta$ and $\partial_y   \YS(0)$ have the same sign (by \cref{partialy YSpositive}), there are  $n^-, m^- >\varepsilon^{-1} $ such  that:
\[n^-\ge \varepsilon^{-2}\cdot \lambda^{m} \qand
\partial_y   \YS(0)\cdot \lambda^{m^-}\sigma_{u}^{n^-}\in   
\Delta+[-\tfrac \varepsilon{10},\tfrac \varepsilon{10}]\; .\]
 Then let $m^+ := m+m^-$ and $n^+:= n+n^-$. This gives $\lambda^{m^\pm}\cdot \partial_y   \YS(0)\cdot \sigma_u^{n^\pm}=\Delta+[-\varepsilon,\varepsilon]$. \end{proof}
 
A consequence of the Lemmas~\ref{preprecondition fro blender}, \ref{precondition fro blender} and~\ref{choixirrationel} is:
 \begin{coro}\label{condition for blender}
For every $\varepsilon>0$ there exist  $n^-,m^-,n^+,m^+$
such that the renormalized maps $\cR g^\pm$ coincide, up to a term that is $C^r$-dominated by $\varepsilon$, with:
\[(x,y)\mapsto(0, \Delta \cdot y)+ (s'_x\; ,\;  \varepsilon^{-1}\cdot  \lambda^{m^+} \cdot s'_y- \varepsilon^{-1}\cdot \lambda ^{m^+}\cdot\sigma_u^{n^\pm } \cdot \partial_y   \YS(0) \cdot q'_y).\]
\end{coro} 
 
 \subsection{
Proof of \cref{PPaffine}: from   strong heterocycles to blenders}\label{Translating candidates}
Let $n^-,m^-,n^+,m^+$ be given by \cref{condition for blender}.  It remains to choose the values of  $s'_y$ and $q_y'$, such that the renormalized maps  $\cR g^\pm$ are $C^r$-close to:
\[(x,y)\mapsto(s'_x, \Delta \cdot y)\pm  (\Delta-1)  .\]
In view of \cref{condition for blender}, it is enough to  ask:
\[\varepsilon^{-1}\cdot  \lambda^{m^+} \cdot s'_y- \varepsilon^{-1}\cdot \lambda ^{m^+}\cdot \sigma_u^{m^\pm } \cdot \partial_y   \YS(0) \cdot q'_y =\pm(\Delta-1)+O(\varepsilon).\]

This is implied by choosing $s'_y$ and $q'_y$ as follows:
\begin{equation}\label{choice-sq}
\varepsilon^{-1}\cdot  \lambda^{m^+} \cdot s'_y= \Delta-1
\qand
\varepsilon^{-1}\cdot \lambda ^{m^+}\sigma_u^{n^-} \cdot \partial_y   \YS(0) \cdot q'_y  =2(\Delta-1)\; .
\end{equation}
Indeed, one has $\sigma_u ^{n^+-n^- }\geq \varepsilon^{-1}$ by~\eqref{e.hypo2},
and with \eqref{e.hypo1}, \eqref{e.hypo3}, \cref{fact-eps}, the choices \eqref{choice-sq} give $s'_y=O(\varepsilon^2)$, $q'_y=O(\varepsilon)$
and $\varepsilon^{-1}\cdot \lambda ^{m^+}\sigma_u^{n^+} \cdot \partial_y   \YS(0) \cdot q'_y  =O(\varepsilon)$.

By \cref{ are blender}, $\{\cR g^+, \cR g^-\}$   defines a nearly affine blender with activation domain containing  $[-2,2]\times [-1/2 ,1/2]$.
Thus, $\{  g^+,   g^-\}$ defines a blender with activation domain containing   $\cH([-2,2]\times [-1/2  ,1/2 ])= [-2,2]\times [-\varepsilon \cdot \lambda^{-m^+} /2, \varepsilon \cdot \lambda^{-m^+}/2 ]$.
Choosing $|\Delta-1|<1/4$, one gets $|s'_y|<\varepsilon \cdot \lambda^{-m^+} /2$
and $S'$ belongs to this activation domain. Also the point $Q\equiv 0$ belongs to this activation domain. Note that   the unstable manifold of $Q$ stretches across $\{s_x'\}\times [-\varepsilon \cdot \lambda^{-m^+} /2, \varepsilon \cdot \lambda^{-m^+}/2 ]$ and so the stable manifolds of the blender. Hence $Q$ is homoclinically related to the blender.
\cref{PPaffine} is proved in the $C^\infty$ case.
\qed

\subsection{Proof of the \cref{PPaffine} in the analytic case}
The whole previous proof is still valid in the analytic setting but 
\cref{unfolding as we want}.
Note that the proof of \cref{PPaffine} does not use that the $r$ first derivatives of $\TS$ and $\TQ$ remain unchanged but only that they are bounded. Thus to prove the analytic case of  \cref{PPaffine}, it suffices to show:
\begin{claim} \label{unfolding as we want_analytic}
For every small numbers $s'_y$ and $q'_y$, there exists a $C^\omega$ perturbation of the dynamics such that the inverse branches $\cS$ and $\cQ$ remain unchanged, while the  continuations  of the inverse branches   $\TS$ and $\TQ$
derivatives at $0$ and
satisfy:
\[\YS(0)=s'_y \qand \YQ(0)=q'_y.\]
Moreover their $C^r$-norm vary continuously with the parameters $s'_y,q'_y$.
\end{claim}
\begin{proof} The perturbation technique follows the same lines as the proof of \cref{l.replace-omega}. First we embed analytically $M$ into $\R^N$, and we define an analytic retraction $\pi$ from a neighborhood of $M\subset \R^N$ to $M$. Then we chose a $C^\infty$-family $(f_p)_{p\in [-\varepsilon, \varepsilon]^8}$  such that $f_0=f$, such that $f_p$ coincide with $f$ outside of a small neighborhood of $\{S, Q\}$, and such that the following map is a local diffeomorphism:
\[\Phi \colon p\in  [-\varepsilon, \varepsilon]^8\mapsto (S_p, P_p, \sigma(p), \lambda(p), \sigma_u (p),\sigma_{uu} (p))\in M^2\times \R^4,\]
where $S_p$ and $P_p$ are the continuations of $S$ and $P$, while 
$(\sigma^{-1}(p), \lambda^{-1}(p))$ and $(\sigma^{-1}_u (p),\sigma^{-1}_{uu} (p))$ are their eigenvalues. Then using Stone-Weierstrass  theorem and the retraction $\pi$, we define an analytic family $(\tilde f_p)_{p\in [-\varepsilon, \varepsilon]^8}$ such that $\tilde f_0=f$ and such that the continuation of $\Phi$  remains a diffeomorphism. We can thus extract from this family a $4$-parameter family $(\tilde f_p)_{p\in [-\varepsilon, \varepsilon]^4}$ along which the eigenvalues are constant, but such that
the continuations $\tilde S_p$ and $\tilde P_p$ of $S$ and $P$ still satisfy that the following map is a local diffeomorphism:
\[p\in  [-\varepsilon, \varepsilon]^4\mapsto (\tilde S_p,\tilde  P_p)\in M^2\; .\]
In \textsection\ref{Setting modulo small perturbation}, we assumed  the eigenvalues of these points to be non-resonant. Thus we can apply \cite{Ta71} which provides $C^r$-families of coordinates at $S$ and $P$ in which $f_p|V_{S}'$ and $f_p|V_Q'$ coincide with diagonalized linear part of $D_{S} f_p$ and  $D_{Q} f_p$, which do not depend on $p$.  Consequently the inverse branches $\cS$ and $\cQ$ (seen in the coordinates) remain unchanged when $p$ varies in $[-\varepsilon, \varepsilon]^4$.  Also the continuations of the inverse branches  $\TS$ and $\TQ$ vary $C^r$-continuously with $p$. On the other hand, the variation of the relative positions of  the continuation of $S$  and $Q$ w.r.t. the local unstable manifold of $Q$ and the strong unstable manifold of $S$ is non-degenerated. 
\end{proof}

\subsection{Proof of \cref{PPPaffine}: from   strong paraheterocycles to parablenders
}\label{sectionparadense}

We now consider a $C^\infty$ family of $(f_a)_{a\in \R}$
and continue to work in the setting of \textsection \ref{Setting modulo small perturbation}--\ref{Translating candidates} for the map $f=f_0$.

The continuations of the periodic points are $(S_a)_{a\in \R}$, $(Q_a)_{a\in \R}$,
with eigenvalues $\sigma_u^{-1}(a),\sigma_{uu}^{-1}(a)$ and $\lambda^{-1}(a),\sigma^{-1}(a)$.
By \cite{Ta71}, their linearizing coordinates can be extended for every $a\in I$ of $I$ sufficiently small, as $C^{r+1}$-family of $C^{r+1}$-diffeomorphisms. 
This enables us to consider the continuations $\cS_a$, $\cQ_a$, $\TQ_a$ and $\TS_a$ of the inverse branches $\cS $, $\cQ $, $\TQ$ and $\TS$.
They are still of the form:
\[\cS_a: (x,y)\in V_S\mapsto   (\sigma_{uu} (a)\cdot  x, \sigma_u (a)\cdot  y),  \quad \cQ_{a}: (x,y)\in V_Q\mapsto   (\sigma (a)\cdot  x, \lambda  (a)\cdot  y),\;\]
 \[\TS_a=(\XS_a,\YS_a): (x,y)\in V''_S \hookrightarrow 
 V_Q,   \quad \TQ_a=(\XQ_a,\YQ_a): (x,y)\in V''_Q\hookrightarrow V_S, \; \]
and they allow to define the preimages by $f_a$:
$$S'_a=(s'_x(a),s'_y(a)):=\TS_a(0) \qand Q'_a=(q'_x(a),q'_y(a)):=\TQ_a(0).$$

Observe that up to a perturbation localized at a neighborhood of $\cS_0$ we can also assume:
\begin{equation}\label{paraasume} \partial_a \frac{\log \sigma_u(a)}{\log \lambda(a)}\neq 0\quad \text{ at } a=0\; .\end{equation}

We consider $\Delta>1$, $\varepsilon>0$, and the integers $n^+, m^+, n^-, m^-$ as before.
This allows to extend the definition of $g^\pm$ as families $(g^\pm_a)_{a\in I}$.
We also extend the rescaling maps:
\[\cH_a:\quad (x,y)\mapsto (x,\varepsilon \cdot \lambda^{m^+}(a)\cdot y),\]
\[\cH^\pm_a:\quad (x,y)\mapsto (x,\varepsilon \cdot \lambda^{m^\pm-m^+}(a)\cdot y)-Q'_a,\]
and, similarly to \cref{def Rg}, the renormalized inverse branches:
  \begin{equation*}
  \cR g^\pm_a:= \cH^{-1}_a\circ g^\pm_a  \circ \cH_a=  \Psi^\pm_a \circ \Phi^\pm_a,\end{equation*}
$$  \text{where }\quad\quad
\Phi^\pm_\alpha :=(\cH^\pm_\alpha)^{-1}\circ  \TQ_\alpha \circ \cQ_\alpha^{m^\pm} \circ \cH_\alpha \qand
\Psi^\pm_\alpha :=(\cH _\alpha)^{-1}\circ  \TS_\alpha \circ \cS_\alpha^{n^\pm} \circ \cH^\pm_\alpha.
$$
For $a$ small,  $\cR g^+_a, \cR g^-_a$ are well defined on $B:=[-1, 1]\times [-2,2] $ by \cref{Rg well def}  and form a $C^r$-nearly affine blender $K_a$
by \cref{Translating candidates}.
We also rescale the parameter space:
 \[\alpha(a):=\Delta_-(a)-\Delta_-(0) \quad \text{ where }\quad  \Delta_\pm(a) :=\lambda^{m^\pm}(a)\cdot \partial_y   \YS_a(0)\cdot \sigma_u^{n^\pm}(a).\]   
\begin{lemma}\label{estime alpha}
\begin{enumerate}
\item The map $\alpha$ is a local diffeomorphism at $a=0$.
\item The function $\alpha\mapsto a(\alpha)$ is $C^r$-dominated by in $1/n^-$ (and hence by $\varepsilon$).
\item The maps $\alpha \mapsto \Delta_\pm(\alpha)-(\Delta+ \alpha )$ are $C^r$-dominated by $\varepsilon$.
\end{enumerate}
\end{lemma}
\begin{proof}
First observe that $\alpha(0)=0$. Then
\begin{multline*}
\Delta_-^{-1} \cdot \partial_a \alpha=
\partial_a  \log \Delta_-  =  \partial_a \log (\lambda^{m^-} \cdot \sigma_u^{n^-} \cdot \partial_y \YS_a(0))
=  \partial_a \left(\frac{\log (\lambda^{m^-} \cdot \sigma_u^{n^-})}{\log \lambda } \cdot \log \lambda+\log \partial_y \YS_a(0)\right)
\\
=  n^- \cdot  \log \lambda\cdot \partial_a \frac{\log \sigma_u}{\log \lambda}  + 
 \frac{m^-\cdot   \log \lambda  
+ n^-  \log \sigma_u}{\log \lambda} \cdot \partial_a \log \lambda  +
\partial_a \log (\partial_y \YS_a(0))\; .
\end{multline*}
Thus by~\eqref{e.hypo3}, when $\varepsilon$ is small, $\partial_a \alpha|_{a=0}$ is invertible, of the order of $n^-$, giving the first item.

By induction, one gets that the higher derivatives can be written as:
\begin{multline}\label{e.higherDelta}
\partial^k_a \Delta_-=\partial^k_a \alpha
=  \Delta_-\cdot \bigg(n^- \cdot  \log \lambda\cdot \partial_a \frac{\log \sigma_u}{\log \lambda}  + 
 \frac{m^-\cdot   \log \lambda  
+ n^-  \log \sigma_u}{\log \lambda} \cdot \partial_a \log \lambda  +\\
\partial_a \log (\partial_y \YS_a(0))\bigg)^k
+\Delta_-\cdot R_k(n^-,m^-)\; ,
\end{multline}
where $R_k(n^-,m^-)$ is a polynomial in $n^-,m^-$ with degree smaller or equal to $k-1$.
Hence $\partial^k_a \alpha$ is dominated by $(n^-)^k$.
Note that $\partial^k_\alpha a\cdot (\partial_a \alpha)^{k+1}$ is a linear combination of terms of the form
$(\partial_a \alpha)^{i_1}\cdot(\partial^2_a \alpha)^{i_2}\cdots (\partial^k_a \alpha)^{i_k}$, where $i_1+2\cdot i_2+\cdots+k\cdot i_k\leq k$.
This implies that $\partial^k_\alpha a$ is dominated by $1/n^-\leq \varepsilon$ as announced in the second item.

The definition of $\alpha$ gives $\Delta_-(\alpha)=\Delta_-(0)+ \alpha$ and $|\Delta_-(0)-\Delta|<\varepsilon$
by~\eqref{e.hypo3}.
In order to get the third item, it is thus enough to prove that each derivative
$\partial^k_\alpha(\Delta_+-\Delta_-)$ is dominated by $\varepsilon$.

By~\eqref{e.hypo3} and \eqref{e.hypo2}, $\Delta_+-\Delta_-$ is dominated by $\varepsilon$,
and $n^+-n^-$ is dominated by $\varepsilon n^-$, whereas $m^-\cdot   \log \lambda  
+ n^-  \log \sigma_u$ and $m^+\cdot   \log \lambda  
+ n^+  \log \sigma_u$ are uniformly bounded. { 
The partial derivative $\partial_a^k\Delta_-$ satisfies Eq.~\eqref{e.higherDelta}.
Replacing $\Delta_-,n_-,m_-$ by $\Delta_+,n_+,m_+$, one obtains a relation for  $\partial_a^k\Delta_+$.
Taking the difference, one concludes that $\partial_a^k(\Delta_+-\Delta_-)$
is dominated by $\varepsilon (n^-)^k$.
Since $\partial_\alpha ^k(\Delta_+-\Delta_-)$
is a linear combination of terms $\partial_a^m(\Delta_+-\Delta_-)\cdot \partial^{i_1}_\alpha a\cdots\partial^{i_\ell}_\alpha a$ with $i_1+\cdots+ i_\ell=m$, by the second equality of 
it is dominated by $\varepsilon$.}
\end{proof}

\begin{coro}\label{c.bounds}
\begin{enumerate}
\item The maps $\alpha \mapsto S'_\alpha-S'_0$ and $\alpha \mapsto Q'_\alpha-Q'_0$ are $C^r$-dominated by $\varepsilon$.
\item The first derivative of $\alpha \mapsto  \varepsilon^{-1}\cdot \lambda ^{m^+-m^\pm }(\alpha)\cdot q'_y(\alpha)$
is $C^{r-1}$-dominated by $\varepsilon$.
\item The first derivative of 
$\alpha\mapsto \varepsilon^{-1}\cdot \lambda ^{m^+}(\alpha)\cdot\sigma_u^{n^\pm }(\alpha)  \cdot  q'_y(\alpha))$ is $C^{r-1}$-dominated by $\varepsilon$.
\end{enumerate}
\end{coro}
\begin{proof}
The first item is a direct consequence of \cref{estime alpha}:
in particular the first derivative of
$\alpha \mapsto q'_y(\alpha)$ is $C^{r-1}$-dominated by $1/n^-$.
By our choice~\eqref{choice-sq}, $q'_y(0)$ is dominated by $\varepsilon\cdot\lambda^{m^--m^+}$.
Similarly, the first derivative of
$\alpha \mapsto \lambda^{m^+-m^-}(\alpha)$ is $C^{r-1}$-dominated by
$$\max\{(\tfrac{m^+-m^-}{n^-})^k\lambda^{m^+-m^-}:\; 1\leq k\leq r\}\leq \tfrac{m^+-m^-}{n^-} \lambda^{m^+-m^-}<\varepsilon^2\lambda^{m^+-m^-} ,$$
using~\eqref{e.hypo2}.
The second item is thus a consequence of~\eqref{e.hypo2}:
$  \varepsilon^{-1} \lambda ^{m^+-m^\pm }/n^-\leq \varepsilon$.

The third item is obtained similarly, by writing $\lambda ^{m^+}(\alpha)\cdot\sigma_u^{n^\pm }(\alpha)=
\lambda ^{m^+-m^\pm}(\alpha)\cdot\lambda^{m^\pm}\sigma_u^{n^\pm }(\alpha)$
and by using~\eqref{e.hypo3}.
\end{proof}

\begin{lemma}\label{condition for parablender1}
With $p_y:(x,y)\mapsto y$, the families of maps $( \Phi^\pm_\alpha-p_y)_\alpha$ are $C^r$-dominated by $\varepsilon$.
\end{lemma}
 \begin{proof} In addition to \cref{preprecondition fro blender}, it remains to study the partial derivatives involving $\alpha$.
 Let us recall that $\Phi^\pm_\alpha(x,y)$ is given by
 $$(  \XQ_\alpha  ,  \varepsilon^{-1}  \cdot  \lambda^{m^+-m^\pm}(\alpha) \cdot  \YQ_\alpha )(\sigma^{m^\pm}(\alpha)  \cdot   x
,\varepsilon  \cdot \lambda^{m^\pm-m^+}(\alpha) \cdot y)- (  q'_x(\alpha)  ,  \varepsilon^{-1}  \cdot  \lambda^{m^+-m^\pm}(\alpha) \cdot  q'_y(\alpha) ).$$
By \cref{c.bounds}, one can reduce to consider the family indexed by $\alpha$ and formed by:
\begin{equation}\label{e.simplify}
(x,y) \mapsto   (  \XQ_\alpha   ,  \varepsilon^{-1}  \cdot  \lambda^{m^+-m^\pm}(\alpha) \cdot  \YQ_\alpha  )(\sigma^{m^\pm}(\alpha)  \cdot 
 x
,\varepsilon  \cdot \lambda^{m^\pm-m^+}(\alpha) \cdot y) +Cst .
\end{equation}
By taking the first derivative w.r.t $\alpha$ and further derivatives w.r.t. $x,y,\alpha$,
factors of the form $(m^\pm)^i(\sigma)^{m^\pm}$, $(m^+-m^-)^i\lambda^{m^+-m^-}$
appear, together with at least one factor of the form $\partial^i_\alpha a$.
Since $(m^\pm)^r(\sigma)^{m^\pm}<1$ and
$(m^+-m^-)^r\lambda^{m^+-m^-}<1$ by \cref{fact-eps} and using \ref{estime alpha},
the derivative of~\eqref{e.simplify} w.r.t. $\alpha$ forms a family $C^{r-1}$-dominated by the map
$\varepsilon^{-1}  \cdot  \lambda^{m^+-m^\pm}(0) \cdot \tfrac1{n^-}$ which is dominated by $ \varepsilon$ using \cref{estime alpha}. 
\end{proof}

\begin{lemma}\label{condition for parablender2}
The families $(\Psi^\pm_\alpha)_\alpha$ coincide, up to the addition of maps $C^r$-dominated by $\varepsilon$, with the families defined by: 
\[((x,y),\alpha)\mapsto(0, (\Delta+\alpha) \cdot y)+ (s'_x(0) ,  \varepsilon^{-1}\cdot  \lambda^{m^+}(\alpha) \cdot s'_y(\alpha)- \varepsilon^{-1}\cdot \lambda ^{m^+}(\alpha)
\cdot \sigma_u^{n^\pm }(\alpha) \cdot \partial_y   \YS_\alpha (0) \cdot q'_y(\alpha)).  \]
\end{lemma}
\begin{proof}
In addition to \cref{precondition fro blender}, we are reduced to examine the $(\partial_\alpha \Psi_\alpha^\pm)_\alpha$. We have:
 \[\Psi^\pm_\alpha(x,y) = ( \XS_\alpha, \varepsilon^{-1} \lambda ^{m^+} (\alpha)\YS_\alpha )(\sigma_{uu}^{n^\pm}(\alpha) \cdot (x-q'_x(\alpha)),  \sigma_{u} ^{n^\pm}(\alpha)
\cdot (\varepsilon \cdot \lambda^{m^\pm-m^+}(\alpha)\cdot y-q'_y(\alpha))) 
\; . \] 
We first discuss the families $(\partial_x\Psi^\pm_\alpha)_\alpha$, $(\partial_y\Psi^\pm_\alpha)_\alpha$ and then the families $(\partial_\alpha\Psi^\pm_\alpha(0))_\alpha$.
\smallskip

\noindent
\emph{Step 1. The families $(\partial_x \Psi_\alpha^\pm)_\alpha$} are controlled as in the proof of~\cref{condition for parablender1}, by bounding the factors $\partial ^k_\alpha a$ by $1/n^-$
with~\cref{estime alpha}.
By~\eqref{e.hypo3}, \eqref{e.hypo2}, $m^-,m^+,n^+$ are dominated by $n^-$.
{ 
All of this implies that $\log(\lambda^{m^\pm})$, $\sigma_u^{n^\pm}$, $\sigma_{uu}^{n^\pm}$,
as functions of $\alpha$, are $C^r$-bounded.
One deduces that $\partial_x \Psi_\alpha^\pm$ are $C^{r-1}$-dominated by
$\sigma_{uu} ^{n^\pm}\cdot \lambda ^{m^+} \cdot \varepsilon^{-1}$.
Arguing as in the proof of~\cref{precondition fro blender}, $\partial_x \Psi_\alpha^\pm$ are thus $C^{r-1}$-dominated by
$$(\tfrac{\sigma_{uu}}{\sigma_u}) ^{n^-}\cdot \lambda^{m^+-m^-} \cdot \varepsilon^{-1}<(\tfrac{\sigma_{uu}}{\sigma_u}) ^{n^-} \cdot n^-\cdot \varepsilon<\varepsilon.$$
}

\noindent
\emph{Step 2. The families $(\partial_y \Psi_\alpha^\pm)_\alpha$}
have a first coordinate which is $C^{r-1}$-dominated by
$(n^\pm)^r\cdot \sigma_{u} ^{n^\pm}\cdot (m^\pm-m^+)^r\cdot \lambda ^{m^\pm-m^+} \cdot \varepsilon$,
and by $\varepsilon$ by \cref{fact-eps}.
The second coordinate of $\partial_y \Psi_\alpha^\pm$ equals:
\[((x,y),\alpha) \mapsto \sigma_u^{n^\pm}(\alpha) \cdot\lambda^{m^\pm}(\alpha)\cdot\partial_y\YS_\alpha \; \bigg(\sigma_{uu}^{n^\pm}(\alpha) \cdot (x-q'_x(\alpha)),  \sigma_{u} ^{n^\pm}(\alpha)
\cdot (\varepsilon \cdot \lambda^{m^\pm-m^+}(\alpha)\cdot y-q'_y(\alpha))\bigg) 
\;. \]
It differs with $(\sigma_u^{n^\pm}(\alpha) \cdot \lambda^{m^\pm } (\alpha)\cdot \partial_y \YS_\alpha(0))_\alpha$
up to a map which is $C^{r-1}$-dominated by
$$\sigma_u^{n^\pm} \cdot\lambda^{m^\pm}\cdot \max\bigg\{ (n^\pm)^r\cdot \sigma_{uu}^{n^\pm}\;,\;
(n^\pm)^r\cdot \sigma_{u} ^{n^\pm}\cdot (m^+-m^\pm)^r 
\cdot \lambda^{m^\pm-m^+}\cdot \varepsilon\bigg\},$$
and hence by $\varepsilon$ from \eqref{e.hypo3} and \cref{fact-eps}.
By definition $\sigma_u^{n^\pm}(\alpha) \cdot \lambda^{m^\pm} (\alpha)\cdot \partial_y \YS_\alpha(0)=
\Delta_\pm(\alpha)$ and $\Delta_\pm(\alpha)$ coincides with $\Delta+\alpha$ up to a term that is $C^r$-dominated by $\varepsilon$, by \cref{estime alpha}.

Up to here, we have shown that the families $(D\Phi^\pm_\alpha)_\alpha$ coincide with the spatial derivative of the map $((x,y),\alpha)\mapsto (0,(\Delta+\alpha)\cdot y)$,
up to a term $C^{r-1}$-dominated by $\varepsilon$.
\smallskip

\noindent
\emph{Step 3.
The families $(\partial_\alpha\Psi^\pm_\alpha(0))_\alpha$,} are given by:
\[\Psi^\pm_\alpha(0) =  ( 
\XS_\alpha, \varepsilon^{-1} \cdot \lambda ^{m^+}(\alpha) \cdot  \YS_\alpha )( - \sigma_{uu} ^{n^\pm}(\alpha)\cdot q'_x (\alpha), -  \sigma_{u} ^{n^\pm}(\alpha)
\cdot  q'_y(\alpha) ) \; . \]
The first coordinate of each derivative $\partial^k_\alpha\Psi^\pm_\alpha(0)$ is dominated by derivatives $\partial^i_\alpha a$, hence the first coordinate of
$\partial_\alpha\Psi^\pm_\alpha(0)$ is dominated by $\varepsilon$ by \cref{estime alpha}.

By similar estimates as in \cref{precondition fro blender}, combined with \cref{estime alpha},
the second coordinate of $\partial_\alpha\Psi^\pm_\alpha(0)$ can be reduced (up to a term $C^{r-1}$-dominated by $\varepsilon$) to:
\[\varepsilon^{-1} \cdot \lambda ^{m^+}(\alpha) \cdot  \YS_\alpha ( 0) \; + \;
\varepsilon^{-1} \cdot \lambda ^{m^+}(\alpha) \cdot  D\YS_\alpha ( 0) .(0, -  \sigma_{u} ^{n^\pm}(\alpha)
\cdot  q'_y(\alpha) ) \;,\]
which is also equal to $\varepsilon^{-1}\cdot  \lambda^{m^+}(\alpha) \cdot s'_y(\alpha)- \varepsilon^{-1}\cdot \lambda ^{m^+}(\alpha)
\cdot \sigma_u^{n^\pm }(\alpha) \cdot \partial_y   \YS_\alpha (0) \cdot q'_y(\alpha)$.
\end{proof}

As a consequence of the Lemmas~\ref{condition for parablender1}, \ref{condition for parablender2} and~\ref{choixirrationel}, we have obtained:
\begin{coro}\label{condition for parablender}
For every $\varepsilon>0$ there exist  $n^+, n^-, m^+, m^-$ such that
the families  $(\cR g^\pm_\alpha)_\alpha$ coincide, up to a term dominated by $\varepsilon$, with the families defined by: 
\[(x,y)\mapsto(0, (\Delta+\alpha) \cdot y)+ (s'_x(0) ,  \varepsilon^{-1}\cdot  \lambda^{m^+}(\alpha) \cdot s'_y(\alpha)- \varepsilon^{-1}\cdot \lambda ^{m^+}(\alpha)\cdot\sigma_u^{n^\pm }(\alpha) \cdot \partial_y   \YS_\alpha (0) \cdot q'_y(\alpha))\; .  \]
\end{coro}
\medskip

\begin{proof}[End of the proof of \cref{PPPaffine}]
Corollaries~\ref{condition for parablender} and \ref{c.bounds} reduce the family $(\cR g^\pm_\alpha)_\alpha$ to:
\[(x,y)\mapsto(0, (\Delta+\alpha) \cdot y)+ (s'_x(0) ,  \varepsilon^{-1}\cdot  \lambda^{m^+}(\alpha) \cdot s'_y(\alpha)+ \varepsilon^{-1}\cdot \lambda ^{m^+}(0)\cdot\sigma_u^{n^\pm }(0)  \cdot \partial_y   \YS_0 (0) \cdot q'_y(0))\; .  \]
As in \cref{Translating candidates},
$$\varepsilon^{-1}\cdot \lambda ^{m^+}(0)\cdot\sigma_u^{n^- }(0) \cdot \partial_y   \YS_0 (0) \cdot q'_y(0)=2(\Delta-1)
\qand \varepsilon^{-1}\cdot \lambda ^{m^+}(0)\cdot\sigma_u^{n^+ }(0) \cdot \partial_y   \YS_0 (0) \cdot q'_y(0)=O(\varepsilon).$$
As we started with a strong $C^r$-paraheterocycle,  all the $r$-first derivatives of $\alpha \mapsto s'_y(\alpha)$ equal $0$ at $0$. So by \cref{unfolding as we want}, we can 
can perturb $(f_a)_a$ so that  $\alpha \mapsto s'_y(\alpha)$ has the same $r$-jet as the $C^r$-small function   $  \alpha\mapsto \varepsilon\cdot  \lambda^{-m^+}(\alpha) \cdot  (\Delta-1) $ at $\alpha=0$. Then we obtain that $(\cR g^\pm_\alpha)_\alpha$  are $\delta$-$C^r$-close to:
\[(x,y)\mapsto(s'_x(0), (\Delta+\alpha) \cdot y\pm (\Delta-1))  \;,  \]
and hence defines a $\delta$-nearly affine $C^r$-parablender,
where $\delta$ is arbitrarily close to $0$ when $\varepsilon\to0$. By \cref{nearly are parablender},
one deduces that the maximal invariant set $K_a$ induced by the maps $g_a^+,g_a^-$ is a $C^r$-parablender.
Its activation domain see in the chart $\cH_\alpha$
contains any germ $\alpha\mapsto z(\alpha)$ with $z(0)\in [-2,2]\times \{0\}$ and $\|\partial_\alpha z(\alpha)\|_{C^{r-1}}\leq \eta$,
where $\eta>0$ is small number independent from $\varepsilon$.
Note that our perturbation satisfies $\cH_\alpha(S'_\alpha)=(s'_x(\alpha),\varepsilon(\Delta-1))$.
Combining with \cref{c.bounds}, item $1$, one concludes that the activation domain of $(K_\alpha)_{\alpha\in I}$
contains the germ of $(S'_\alpha)_{\alpha}$, and the germ of the source $(S_\alpha)$ at $\alpha=0$.
We also recall that  $Q$ is homoclinicaly related to the (para)-blender.
\cref{PPPaffine} is proved.
\end{proof}

\begin{remark}\label{H4}
For each point $\underline x\in \overleftarrow K$, let $\gamma_{\underline x}$ be the unstable curve of $\underline x$
which is a graph over $[-2,2]$. The activation domain is obtain by considering the local unstable manifolds
of the form $(\TS)^{-1}(\gamma_{\underline x})$.
By assumption~\eqref{transversality strong hetro}, $W^u(Q)$ is transverse to $E^{cu}(S)$.
One deduces that the family of local unstable manifolds defining the activation domain of the parablender
satisfies the property announced in \cref{r.H4}.
\end{remark}

\bibliographystyle{alpha-like}
\bibliography{references}

\bigskip
\bigskip

\hspace{-3.2cm}
\footnotesize
\begin{tabular}{l l l l l}
\emph{\normalsize Pierre Berger}
& &
\emph{\normalsize Sylvain Crovisier}
& &
\emph{\normalsize Enrique Pujals}
\\
\texttt{pierre.berger@imj-prg.fr}
&&
\texttt{sylvain.crovisier@math.u-psud.fr}
&&
\texttt{epujals@gc.cuny.edu}
\\
Institut de Math. Jussieu-Paris Rive Gauche
&&
Laboratoire de Math\'ematiques d'Orsay
&& Graduate Center-CUNY
\\
Sorbonne Universit\'e, Univ. de Paris, CNRS,
&&
CNRS - UMR 8628, Univ. Paris-Saclay
&& New York, USA\\
F-75005 Paris, France
&&
Orsay 91405, France
&&
\end{tabular}

\end{document}